\documentclass[12pt,reqno]{amsart}

\usepackage[T1]{fontenc}
\usepackage{lmodern}
\usepackage{microtype}
\usepackage[margin=1in]{geometry}
\usepackage{setspace}
\usepackage{amsmath,amssymb,mathtools}
\usepackage{xcolor}
\usepackage[hidelinks]{hyperref}
\allowdisplaybreaks
\newif\ifanonymous
\anonymousfalse 

\ifanonymous
  \author{Anonymous Author}
\else
  \author{Maotuo Guo}
  \address{School of Science, Harbin University of Science and Technology}
  \email{18813122370@163.com}

  \author{Wendong Wang}
  \address{School of Mathematical Sciences, Dalian University of Technology, Dalian 116024, P. R. China}
  \email{wendong@dlut.edu.cn}

  \author{Shiyang Xiong}
  \address{Institute of Mathematics, Academy of Mathematics and Systems Science, Chinese Academy of Sciences, Beijing 100190, P. R. China}
  \email{xiongshiyang@amss.ac.cn}
\fi

\title[One-Component Regularity in Critical Chemin--Lerner Spaces]{A Critical Chemin--Lerner Regularity Criterion via One Velocity Component for the Three-Dimensional Navier--Stokes Equations}

\hypersetup{
  pdftitle={A Critical Chemin--Lerner Regularity Criterion via One Velocity Component for the Three-Dimensional Navier--Stokes Equations},
  pdfauthor={Maotuo Guo, Wendong Wang, Shiyang Xiong},
  pdfsubject={A critical one-component regularity criterion for the three-dimensional Navier--Stokes equations in Chemin--Lerner spaces},
  pdfkeywords={three-dimensional Navier--Stokes equations, one-component regularity, suitable weak solutions, critical Chemin--Lerner spaces, type-I local energy, pressure decomposition}
}

\numberwithin{equation}{section}

\newtheorem{theorem}{Theorem}[section]
\newtheorem{proposition}[theorem]{Proposition}
\newtheorem{lemma}[theorem]{Lemma}
\newtheorem{corollary}[theorem]{Corollary}
\theoremstyle{remark}
\newtheorem{remark}[theorem]{Remark}

\newcommand{\R}{\mathbb R}
\newcommand{\Z}{\mathbb Z}
\newcommand{\dd}{\,\mathrm d}
\newcommand{\eps}{\varepsilon}
\newcommand{\Bdot}{\dot B}
\newcommand{\DeltaDot}{\dot\Delta}
\newcommand{\Sdot}{\mathcal S'_h}

\newcommand{\esssup}{\mathop{\mathrm{ess\,sup}}}
\newcommand{\norm}[2][]{\left\lVert #2\right\rVert_{#1}}
\newcommand{\abs}[1]{\left\lvert #1\right\rvert}
\newcommand{\rev}[1]{{\color{black}#1}}
\newenvironment{revision}{\par\begingroup\color{black}}{\par\endgroup}

\newcommand{\paperabstract}{%
{\color{black}
We prove a scaling-critical regularity criterion involving only one velocity component for finite-energy suitable weak solutions of the three-dimensional incompressible Navier--Stokes equations.  Let $2<p<\infty$ and $m=3p/(p-2)$, so that $2/p+3/m=1$.  We show that a singularity cannot occur provided
\[
 \sum_{j\in\mathbb Z}
 \|\dot\Delta_j u^3\|_{L^p(0,T;L^m(\mathbb R^3))}<\infty,
\]
that is, $u^3\in\widetilde L^p(0,T;\dot B^0_{m,1}(\mathbb R^3))$.  The assumption is a spatial-frequency $\ell^1$ refinement of the still unresolved critical condition $u^3\in L^p_tL^m_x$ and is complementary to the Lorentz-in-time refinement $L^{p,1}_tL^m_x$ obtained by Wang, Wu, and Zhang.  By Bernstein embedding, the result extends to $u^3\in\widetilde L^p_t\dot B^s_{q,1}$ on the nonnegative-smoothness critical line $s=-1+2/p+3/q\ge0$.

The principal innovation is a frequency--scale matching scheme
embedded in the local energy inequality. Each dyadic block of
$u^3$ is retained until it is paired with the vertical scale
selected by a one-dimensional backward heat kernel. Low
frequencies gain from the slab thickness, high frequencies from
transferring the projection to the localized flux and applying
an inverse Bernstein estimate, and spatially separated pressure
sources from harmonic decay. These mechanisms generate a
two-sided $\ell^1$ kernel, converting spatial-frequency
summability into summability of physical-scale energy increments.
Consequently, we obtain a uniform Type-I local energy bound
without a Lorentz refinement in time; compactness and
one-component rigidity then exclude singular blow-up limits
whose third velocity component vanishes.

}
}
\begin{document}
\raggedbottom

\begin{abstract}\paperabstract\end{abstract}
\maketitle
\section{Introduction and main results}\label{sec:intro}

Consider the three-dimensional incompressible Navier--Stokes equations as follows:
\begin{equation}\label{eq:NS}
 \partial_tu+u\cdot\nabla u-\Delta u+\nabla\pi=0,
 \qquad \nabla\cdot u=0
 \quad\text{in }\mathbb R^3\times(0,T),
\end{equation}
where $u=(u^1,u^2,u^3)$ is the velocity and $\pi$ is the pressure.  Leray and Hopf constructed global finite-energy weak solutions, but their regularity and uniqueness in three dimensions remain open \cite{Leray1934,Hopf1951}.  The classical Prodi--Serrin theory asserts regularity when the full velocity satisfies
\begin{equation}\label{eq:LPS-intro}
 u\in L^p(0,T;L^m(\mathbb R^3)),
 \qquad \frac2p+\frac3m\le1,
 \qquad 3<m\le\infty,
\end{equation}
with the limiting space $L^\infty_tL^3_x$ treated by Escauriaza, Seregin, and \v{S}ver\'ak \cite{Prodi1959,Serrin1962,Giga1986,ESS2003}.  Critical-space continuation criteria have since been developed through mild-solution theory, backward uniqueness, profile decompositions, and critical Besov methods; see, among others, \cite{Kato1984,KochTataru2001,KenigKoch2011,Seregin2012,GallagherKochPlanchon2013,Albritton2018}.

A second line of research concerns local regularity of suitable weak solutions.  Scheffer \cite{Scheffer1976} initiated the partial-regularity theory, and Caffarelli, Kohn, and Nirenberg proved that the one-dimensional parabolic Hausdorff measure of the singular set is zero \cite{CKN1982}.  Subsequent work clarified the role of scale-invariant local energy, velocity, and pressure quantities and provided flexible $\varepsilon$-regularity criteria \cite{Lin1998,LadyzhenskayaSeregin1999,Seregin2007Morrey,Seregin2007Local,GustafsonKangTsai2007,Vasseur2007,WangZhang2014,Wolf2015,JiuWangZhou2019}.  Local pressure decompositions are particularly important when the available information is componentwise or anisotropic \cite{Wolf2017}.
Regularity criteria involving only part of the velocity field are substantially more delicate than \eqref{eq:LPS-intro}.  Controlling $u^3$ does not directly control the horizontal velocity $u^h=(u^1,u^2)$, while the pressure remains coupled to all components through
\[
 -\Delta\pi=\sum_{i,j=1}^3\partial_i\partial_j(u^iu^j).
\]
Early one-component criteria were obtained in subcritical mixed-norm regimes by Neustupa and Penel, Neustupa--Novotn\'y--Penel, Kukavica and Ziane, Cao and Titi, and Zhou and Pokorn\'y \cite{NeustupaPenel1999,NeustupaNovotnyPenel2002,KukavicaZiane2006,CaoTiti2008,ZhouPokorny2010,Neustupa2018}.  Closely related criteria involving one directional derivative, one vorticity component, or one entry of the velocity-gradient tensor were developed in \cite{BeiraoDaVeiga1995,ChaeChoe1999,PenelPokorny2004,KukavicaZiane2007,ZhouPokorny2009,CaoTiti2011,FangQian2013}.

\begin{revision}
The distinction between one- and two-component information is already visible in local suitable-solution theory.  Wang,  Zhang, and  Zhang obtained interior criteria involving the two horizontal components, including anisotropic scaled norms on cylinders of the form $B_r^2\times\mathbb R\times(-r^2,0)$ in the range $1\le3/m+2/p\le2$ \cite{WangZhangZhang2018}.  Their argument exploits the fact that a two-dimensional backward heat kernel converts the principal convection flux into a term containing $u^h$.  The present paper addresses a different problem: only one component is assumed regular, the hypothesis is globally critical rather than supercritical Morrey smallness, and the missing decay is recovered by retaining the spatial frequency until it is paired with the localized flux.
\end{revision}
Scaling-critical one-component criteria were first established in continuation frameworks for strong or mild solutions.  Chemin and Zhang introduced an anisotropic vorticity method and proved a critical condition involving $u\cdot e$ in $L^p_t\dot H^{1/2+2/p}_x$ for $4<p<6$ \cite{CheminZhang2016}.  Chemin, Zhang, and Zhang extended the range to $4<p<\infty$, and Han, Lei, Li, and Zhao completed the range $2\le p<\infty$ in the corresponding mild-solution setting \cite{CheminZhangZhang2017,HanLeiLiZhao2019}.  Related borderline criteria and a time-dependent distinguished direction were studied in \cite{BaeKang2019,Skalak2019,LiuZhang2024}.  These results are genuinely critical and frequency sensitive, but they are formulated for strong or mild solutions and use the anisotropic vorticity structure and the initial-data hypotheses specific to those works.

For suitable weak solutions, Chae and Wolf proved the strict Serrin-range criterion
\[
 u^3\in L^p(0,T;L^m(\mathbb R^3)),
 \qquad \frac2p+\frac3m<1,
\]
and obtained a logarithmically improved borderline condition \cite{ChaeWolf2021}.  Wang, Wu, and Z. Zhang subsequently reached the scaling line by strengthening the time summability to the Lorentz space $L^{p,1}$ \cite{WangWuZhang2024}.  Other local results combine smallness of one component with an a priori scale-invariant bound for the full solution \cite{KukavicaRusinZiane2017,KangNguyen2023,Yu2026}.  Nevertheless, the natural endpoint assertion
\begin{equation}\label{eq:open-strong-time}
 u^3\in L^p(0,T;L^m(\mathbb R^3)),
 \qquad \frac2p+\frac3m=1,
 \qquad 3<m<\infty,
\end{equation}
remains open for general suitable weak solutions.

\begin{revision}
Our purpose is to isolate a spatial-frequency refinement of \eqref{eq:open-strong-time} that is compatible with the local energy inequality.  Let $(\dot\Delta_j)_{j\in\mathbb Z}$ be a homogeneous Littlewood--Paley decomposition and define
\begin{equation}\label{eq:CLnorm-intro}
 \|f\|_{\widetilde L^p(0,T;\dot B^s_{q,1})}
 :=\sum_{j\in\mathbb Z}2^{js}
 \|\dot\Delta_jf\|_{L^p(0,T;L^q(\mathbb R^3))}.
\end{equation}
Thus the time norm is taken before the dyadic $\ell^1$ summation.  At $s=0$ and $2/p+3/m=1$, this norm is invariant under the Navier--Stokes scaling.  It implies the ordinary critical bound $L^p_tL^m_x$, but the extra information is spatial rather than temporal.  This is the structural feature that permits the frequency--observation-scale summation below.
\end{revision}


Our main result is stated as follows.
\begin{theorem}[Zero-smoothness one-component criterion]\label{thm:main}
Let $2<p<\infty$ and
\begin{equation}\label{eq:endpoint-exponents-intro}
 m:=\frac{3p}{p-2}>3,
 \qquad \frac2p+\frac3m=1.
\end{equation}
Let $(u,\pi)$ be a suitable weak solution of \eqref{eq:NS} in $\mathbb R^3\times(0,T)$ such that
\[
 u\in L^\infty(0,T;L^2(\mathbb R^3))\cap L^2(0,T;\dot H^1(\mathbb R^3)),
 \qquad \pi\in L^{3/2}_{\mathrm{loc}}(\mathbb R^3\times(0,T)).
\]
Assume that
\begin{equation}\label{eq:main-assumption}
 u^3\in\widetilde L^p\bigl(0,T;\dot B^0_{m,1}(\mathbb R^3)\bigr),
 \quad\text{equivalently}\quad
 \sum_{j\in\mathbb Z}
 \|\dot\Delta_ju^3\|_{L^p(0,T;L^m)}<\infty.
\end{equation}
Then every $z_0=(x_0,t_0)\in\mathbb R^3\times(0,T]$ is regular.  In particular,
\[
 u\in C^\infty_{\mathrm{loc}}(\mathbb R^3\times(0,T]).
\]
\end{theorem}

A terminal point $(x_0,T)$ is called regular if $u$ is essentially bounded in $B_r(x_0)\times(T-r^2,T)$ for some $r>0$.  Standard parabolic regularity then gives smoothness in every smaller backward cylinder.

\begin{revision}
\begin{remark}[A new critical one-component theory]\label{rem:comparison}
Theorem~\ref{thm:main} should be viewed as complementary to, rather than as a direct strengthening of, the known critical criteria.
\begin{enumerate}
\item Wang, Wu, and Zhang assume $u^3\in L^{p,1}(0,T;L^m)$ on the same Serrin scaling line \cite{WangWuZhang2024}.  Their hypothesis strengthens summability in time and uses the atomic decomposition of $L^{p,1}$ to obtain an $\ell^1$ sequence over physical scales.  Our hypothesis retains the strong $L^p$ time norm and instead imposes $\ell^1$ summability over spatial frequencies.  Both spaces embed into the unresolved class $L^p_tL^m_x$, but neither critical refinement contains the other in general.  At the level of scalar tensor products, $g(t)\phi(x)$ with $g\in L^p\setminus L^{p,1}$ and $\phi\in L^2\cap\dot B^0_{m,1}$ satisfies the present condition but not the Lorentz-time condition; conversely, $h(t)\psi(x)$ with $h\in L^{p,1}$ and $\psi\in L^2\cap L^m\setminus\dot B^0_{m,1}$ has the opposite behavior.  These examples only demonstrate non-comparability of the underlying function spaces; they are not asserted to be Navier--Stokes solutions.
\item The works of Chemin--Zhang, Chemin--Zhang--Zhang, and Han--Lei--Li--Zhao establish critical continuation criteria in the $L^2$-based spaces $L^p_t\dot H^{1/2+2/p}_x$ for strong or mild solutions \cite{CheminZhang2016,CheminZhangZhang2017,HanLeiLiZhao2019}.  Their mechanism is based on anisotropic vorticity estimates.  The present criterion is instead formulated directly for suitable weak solutions, uses the zero-smoothness $L^m$-based Besov space, and does not assume an auxiliary vorticity norm.  The function spaces and the solution frameworks are not ordered, so the results are best regarded as distinct critical realizations of the one-component principle.  Liu and Zhang further allow the controlled direction to vary with time in a local strong-solution framework \cite{LiuZhang2024}; here the direction is fixed, but the conclusion applies to suitable weak solutions.
\item The two-component result of Wang--Zhang--Zhang permits a wider, even supercritical, range of scaled norms on anisotropic cylinders, but requires information on $u^h$ and vanishing of an anisotropic scaled norm across physical scales \cite{WangZhangZhang2018}.  Theorem~\ref{thm:main} uses only $u^3$ and works at the global critical scale, with spatial frequency summability replacing two-component control.
\end{enumerate}
\end{remark}
\end{revision}

The nonnegative-smoothness critical line follows by Bernstein embedding.

\begin{lemma}[Critical-line embedding]\label{lem:critical-line-embedding}
Let $1\le p\le\infty$, $1\le q\le m<\infty$, and $s=3/q-3/m$.  Then
\begin{equation}\label{eq:critical-line-embedding}
 \|f\|_{\widetilde L^p(I;\dot B^0_{m,1})}
 \le C\|f\|_{\widetilde L^p(I;\dot B^s_{q,1})}
\end{equation}
for every time interval $I$.
\end{lemma}

\begin{proof}
Bernstein's inequality gives
\[
 \|\dot\Delta_jf\|_{L^m(\mathbb R^3)}
 \le C2^{3j(1/q-1/m)}\|\dot\Delta_jf\|_{L^q(\mathbb R^3)}
 =C2^{js}\|\dot\Delta_jf\|_{L^q(\mathbb R^3)}.
\]
Taking the $L^p$ norm in time and summing over $j$ proves \eqref{eq:critical-line-embedding}.
\end{proof}

\begin{corollary}[Nonnegative-smoothness critical line]\label{cor:critical-line}
Let $2<p<\infty$, $m=3p/(p-2)$, \rev{$1\le q\le m$}, and
\begin{equation}\label{eq:m-s}
 s:=-1+\frac2p+\frac3q=\frac3q-\frac3m\ge0.
\end{equation}
Under the solution assumptions of Theorem~\ref{thm:main}, if
\begin{equation}\label{eq:critical-line-assumption}
 u^3\in\widetilde L^p\bigl(0,T;\dot B^s_{q,1}(\mathbb R^3)\bigr),
\end{equation}
then the conclusion of Theorem~\ref{thm:main} holds.
\end{corollary}

For $1\le q\le\infty$, let $L^q_\sigma(\mathbb R^3)$ denote the divergence-free subspace of $L^q(\mathbb R^3)$.  We also obtain the corresponding continuation criterion.

\begin{corollary}[Continuation of maximal mild solutions]\label{cor:mild}
Let $u_0\in L^2_\sigma(\mathbb R^3)\cap L^3_\sigma(\mathbb R^3)$ and let $u$ be the maximal $L^3$-mild solution of \eqref{eq:NS} on $[0,T^*)$.  If $T^*<\infty$, then for every $2<p<\infty$, with $m=3p/(p-2)$, every \rev{$1\le q\le m$}, and $s=-1+2/p+3/q$,
\begin{equation}\label{eq:blowup-norm}
 \|u^3\|_{\widetilde L^p(0,T^*;\dot B^s_{q,1})}=\infty.
\end{equation}
\end{corollary}

\begin{revision}
\paragraph{Proof strategy and main analytic point.}
The local energy inequality is tested against a one-dimensional backward heat kernel, which observes a vertical slab at scale $r_k$.  A naive conversion of the endpoint Besov norm into physical-scale component norms would count every high-frequency block at all coarser scales and lose summability.  We therefore keep the Littlewood--Paley index $j$, the pressure-source scale, and the observation scale $k$ separate until the corresponding flux is estimated.  If $j\ll k$, the thickness of the slab yields the factor $2^{-(k-j)/m}$.  If $j\gg k$, the dyadic projector is transferred to a compactly supported flux and an inverse Bernstein estimate yields $2^{-(j-k)}$.  For pressure sources separated from the observation slab, harmonicity supplies a further factor $2^{-\gamma(k-j)}$.  These kernels are summable and lead to a recursive bound for the scaled slab energy.  The resulting type-I estimate is converted into regularity by blow-up compactness: absolute continuity of the critical mixed norm forces the third component of the blow-up sequence to vanish, while the limiting zero-component suitable solution is regular.
\end{revision}

The paper is organized as follows. 
Section~\ref{sec:prelim} fixes notation and records the reconstruction, pressure-normalization, and interpolation tools.  Section~\ref{sec:endpoint} proves the endpoint type-I estimate.  Section~\ref{sec:global} establishes the compactness--rigidity argument and proves Theorem~\ref{thm:main} and its corollaries.

\section{Preliminaries}\label{sec:prelim}

Throughout the paper, $c>0$ denotes a constant that may change from line to line; dependence on fixed parameters is indicated by subscripts.  For $1<q<\infty$, $q'$ denotes the H\"older conjugate of $q$.  We write $x=(x_h,x_3)\in\R^2\times\R$.  For $z_0=(x_0,t_0)$ and $r>0$, set
\[
 B_r(x_0):=\{x\in\R^3:\abs{x-x_0}<r\},
 \qquad
 Q_r(z_0):=B_r(x_0)\times(t_0-r^2,t_0).
\]
When $z_0=(0,0)$, we write $B_r:=B_r(0)$ and $Q_r:=Q_r(0,0)$.

A pair $(u,\pi)$ is a suitable weak solution if it solves \eqref{eq:NS} distributionally, belongs locally to the energy class, and satisfies, for almost every $t_1<t$ and every nonnegative $\phi\in C_c^\infty(\R^3\times[t_1,t])$,
\begin{align}\label{eq:LEI}
 &\frac12\int_{\R^3}|u(x,t)|^2\phi(x,t)\dd x
 +\int_{t_1}^{t}\int_{\R^3}|\nabla u|^2\phi\dd x\dd \tau\notag\\
 &\quad\le \frac12\int_{\R^3}|u(x,t_1)|^2\phi(x,t_1)\dd x
 +\frac12\int_{t_1}^{t}\int_{\R^3}
 |u|^2(\partial_\tau+\Delta)\phi\dd x\dd \tau\notag\\
 &\qquad+\frac12\int_{t_1}^{t}\int_{\R^3}
 (|u|^2+2\pi)u\cdot\nabla\phi\dd x\dd \tau.
\end{align}
The initial term is absent whenever $\phi(\cdot,t_1)=0$.

Choose a standard homogeneous dyadic partition and write
\[
 \DeltaDot_j f:=\varphi(2^{-j}D)f,
 \qquad j\in\Z.
\]
We use the standard homogeneous distribution space
\[
 \Sdot:=\left\{f\in\mathcal S'(\mathbb R^3):
 \lim_{j\to-\infty}\dot S_jf=0
 \text{ in }\mathcal S'(\mathbb R^3)\right\};
\]
see \cite{BCD2011}.  The following elementary reconstruction statement will be used repeatedly.

\begin{revision}
\begin{lemma}[Homogeneous reconstruction in the energy class]\label{lem:homogeneous-reconstruction}
Let $I\subset\mathbb R$ be a finite interval and let $1\le a,b<\infty$.  Suppose that
\[
 f\in L^2(I\times\mathbb R^3),
 \qquad
 \sum_{j\in\mathbb Z}
 \|\dot\Delta_jf\|_{L^a(I;L^b(\mathbb R^3))}<\infty.
\]
Then
\[
 f=\sum_{j\in\mathbb Z}\dot\Delta_jf
 \quad\text{in }\mathcal D'(I\times\mathbb R^3)
 \quad\text{and in }L^a(I;L^b(\mathbb R^3)).
\]
In particular,
\begin{equation}\label{eq:homogeneous-reconstruction-bound}
 \|f\|_{L^a(I;L^b)}
 \le\sum_{j\in\mathbb Z}
 \|\dot\Delta_jf\|_{L^a(I;L^b)}.
\end{equation}
\end{lemma}

\begin{proof}
Set $f_N:=\sum_{|j|\le N}\dot\Delta_jf$.  The assumed absolute summability makes $(f_N)$ a Cauchy sequence in $L^a(I;L^b)$, so $f_N\to g$ in that space for some $g$.  On the other hand, the homogeneous Littlewood--Paley partition reconstructs every $L^2$ function: by Plancherel's theorem, $f_N\to f$ in $L^2(I\times\mathbb R^3)$.  Both convergences imply convergence in distributions, hence $g=f$.  The triangle inequality yields \eqref{eq:homogeneous-reconstruction-bound}.  Thus, in the energy class, the $L^2$ representative removes the usual polynomial ambiguity of the homogeneous decomposition.
\end{proof}
\end{revision}

For $i\in\{1,2,3\}$, $\mathcal R_i$ denotes the Riesz transform with Fourier multiplier $-\mathrm i\xi_i/|\xi|$, and $\mathbb P$ denotes the Helmholtz--Leray projection.  We use the standard Bernstein estimates: if the Fourier support of $f_j$ is contained in a ball of radius $c2^j$ and $1\le q\le m\le\infty$, then
\begin{align}
 \norm[L^m(\R^3)]{f_j}
 &\le c2^{3j(1/q-1/m)}\norm[L^q(\R^3)]{f_j},\label{eq:Bernstein-3d}\\
 \norm[L^m(\R^2_{x_h};L^\infty(\R_{x_3}))]{f_j}
 &\le c2^{2j(1/q-1/m)}2^{j/q}\norm[L^q(\R^3)]{f_j}.
 \label{eq:Bernstein-mixed}
\end{align}

\begin{revision}
\begin{lemma}[Whole-space pressure normalization]\label{lem:whole-space-pressure}
Let $I\subset\mathbb R$ be a finite interval and let $(u,\pi)$ be a suitable weak solution on $\mathbb R^3\times I$ satisfying
\[
 u\in L^\infty(I;L^2(\mathbb R^3))\cap L^2(I;\dot H^1(\mathbb R^3)).
\]
Define
\begin{equation}\label{eq:Riesz-pressure-general}
 \pi_{\rm R}:=\sum_{i,j=1}^3\mathcal R_i\mathcal R_j(u^iu^j).
\end{equation}
Then $\pi_{\rm R}\in L^2(I;L^{3/2}(\mathbb R^3))$, and there exists a function $c\in L^{3/2}_{\rm loc}(I)$ such that
\[
 \pi=\pi_{\rm R}+c(t)
 \quad\text{in }\mathcal D'(\mathbb R^3\times I).
\]
Moreover, replacing $\pi$ by $\pi_{\rm R}$ leaves the local energy inequality unchanged.
\end{lemma}

\begin{proof}
Interpolation between $L^\infty_tL^2_x$ and $L^2_tL^6_x$ gives $u\in L^4(I;L^3)$; hence $u\otimes u\in L^2(I;L^{3/2})$, and Calder\'on--Zygmund boundedness yields the asserted estimate for $\pi_{\rm R}$.  Apply the complementary Helmholtz projection $I-\mathbb P$ to
\[
 \partial_tu-\Delta u+\operatorname{div}(u\otimes u)+\nabla\pi=0.
\]
The first two terms are divergence free, while
\[
 (I-\mathbb P)\operatorname{div}(u\otimes u)=-\nabla\pi_{\rm R}.
\]
Consequently, $\nabla_x(\pi-\pi_{\rm R})=0$ in distributions, so $\pi-\pi_{\rm R}=c(t)$ for a distribution depending only on time.  This distribution is represented by a locally $L^{3/2}$ function.  Indeed, if $J\Subset I$ and $B\Subset\mathbb R^3$ is any fixed ball, then for almost every $t\in J$,
\[
 c(t)=\frac1{|B|}\int_B\bigl(\pi-\pi_{\rm R}\bigr)(x,t)\,dx,
\]
and hence $c\in L^{3/2}(J)$.  Finally, for every compactly supported space--time test function $\phi$,
\[
 \int_I\int_{\mathbb R^3}c(t)u(x,t)\cdot\nabla\phi(x,t)\,dx\,dt
 =-\int_I c(t)\int_{\mathbb R^3}\phi\,\nabla\cdot u\,dx\,dt=0.
\]
Thus the pressure flux in the local energy inequality is invariant under this normalization.
\end{proof}
\end{revision}

For $n\ge0$, set $r_n:=2^{-n}$ and define the one-dimensional backward heat kernel
\begin{equation}\label{eq:Phi}
 \Phi_n(x_3,t)
 :=\frac{1}{\sqrt{4\pi(-t+r_n^2)}}
 \exp\left(-\frac{x_3^2}{4(-t+r_n^2)}\right),
 \qquad t<0.
\end{equation}
Then $(\partial_t+\partial_3^2)\Phi_n=0$.

\section{The zero-smoothness endpoint}\label{sec:endpoint}

We now assume
\begin{equation}\label{eq:endpoint-parameters}
 2<p<\infty,
 \qquad m:=\frac{3p}{p-2}>3,
 \qquad \frac2p+\frac3m=1,
\end{equation}
and work, after translation and a dyadic scaling, on $\R^3\times(-2,0)$.
\begin{revision}
Energy interpolation gives $u\in L^4(-2,0;L^3)$ and $u\otimes u\in L^2(-2,0;L^{3/2})$.  Lemma~\ref{lem:whole-space-pressure} therefore permits us to replace the suitable pressure by its whole-space Riesz representative
\begin{equation}\label{eq:endpoint-pressure-normalization}
 \pi=\sum_{a,b=1}^3\mathcal R_a\mathcal R_b(u^au^b)
 \in L^2(-2,0;L^{3/2}(\mathbb R^3)).
\end{equation}
All local energy inequalities below are unchanged by this normalization.
\end{revision}

Set
\begin{equation}\label{eq:endpoint-frequency-coefficients}
 a_j:=\norm[L^p(-2,0;L^m(\R^3))]{\DeltaDot_j u^3},
 \qquad j\in\Z.
\end{equation}
The endpoint hypothesis is $\sum_{j\in\Z}a_j<\infty$.  Since $u^3\in L^2((-2,0)\times\mathbb R^3)$, Lemma~\ref{lem:homogeneous-reconstruction} identifies the absolutely convergent dyadic series with the original component $u^3$ in $L^p_tL^m_x$ and in distributions.

For every $k\in\Z$, let
\begin{equation}\label{eq:endpoint-slabs}
 r_k:=2^{-k},
 \qquad I_k:=\bigl(-2\min\{r_k^2,1\},0\bigr),
 \qquad S_k:=\R^2\times(-2r_k,2r_k),
 \qquad \mathcal Q_k:=S_k\times I_k.
\end{equation}
Thus negative time indices are automatically truncated to the interval $(-2,0)$ on which the normalized solution is defined; for $k\ge0$ this agrees with $I_k=(-2r_k^2,0)$.
For $k\ge0$, define
\begin{align}
 E_k
 &:=\esssup_{t\in I_k}\int_{S_k}|u(x,t)|^2\dd x
   +\int_{\mathcal Q_k}|\nabla u|^2\dd x\dd t,
 \label{eq:endpoint-energy}\\
 e_k&:=r_k^{-1}E_k.
 \label{eq:endpoint-scaled-energy}
\end{align}
We also put
\begin{equation}\label{eq:endpoint-global-energy}
 E_{\mathrm g}:=
 \esssup_{-2<t<0}\int_{\R^3}|u(x,t)|^2\dd x
 +\int_{-2}^{0}\int_{\R^3}|\nabla u|^2\dd x\dd t.
\end{equation}
For notational convenience, $E_k=e_k=E_{\mathrm g}$ when $k<0$.  This convention is used only across fixed index shifts.  In particular, for every fixed integer $c\ge0$ and $j\ge0$,
\begin{equation}\label{eq:endpoint-negative-index-bound}
 r_j^{-1}E_{j-c}\le 2^c e_{j-c}.
\end{equation}

\begin{lemma}[Slab interpolation and product estimates]\label{lem:endpoint-products}
Let $k\ge0$, $2\le p_1\le\infty$, $2\le q_1\le6$, and $2/p_1+3/q_1=3/2$.  For every fixed integer $c_0\ge0$,
\begin{equation}\label{eq:endpoint-slab-interpolation}
 \norm[L^{p_1}(I_k;L^{q_1}(S_k))]{u}^2
 \le C_{c_0} E_{k-c_0}.
\end{equation}
If $\chi_k=\chi_k(x_3,t)\in C_c^\infty(\R\times(-2,0])$ satisfies
\[
 0\le\chi_k\le1,
 \qquad \operatorname{supp}\chi_k\subset\mathcal Q_{k-c_0},
 \qquad |\partial_3\chi_k|\le c r_k^{-1},
\]
then
\begin{align}
 \norm[L^{p'}(-2,0;L^{m'}(\R^3))]{\chi_k u\otimes u}
 &\le c r_kE_{k-c_0},
 \label{eq:endpoint-product-zero}\\
 \norm[L^{p'}(-2,0;L^{m'}(\R^3))]{\nabla(\chi_k u\otimes u)}
 &\le c E_{k-c_0}.
 \label{eq:endpoint-product-one}
\end{align}
The products are extended by zero outside the support of $\chi_k$, so the norms in \eqref{eq:endpoint-product-zero}--\eqref{eq:endpoint-product-one} are full-space norms.  The corresponding estimates without a spatial cutoff hold on $I_k\times\R^3$ with $E_{k-c_0}$ replaced by $E_{\mathrm g}$.
\end{lemma}

\begin{proof}
\begin{revision}
We first justify the slab estimate with constants independent of $k$.  By scaling a fixed extension operator on $\mathbb R^2\times(-2,2)$, there is a linear map
\[
 \mathcal E_k:H^1(S_k)\longrightarrow H^1(\mathbb R^3)
\]
such that
\begin{align*}
 \|\mathcal E_k f\|_{L^2(\mathbb R^3)}
 &\le C\|f\|_{L^2(S_k)},\\
 \|\nabla\mathcal E_k f\|_{L^2(\mathbb R^3)}
 &\le C\bigl(\|\nabla f\|_{L^2(S_k)}
 +r_k^{-1}\|f\|_{L^2(S_k)}\bigr).
\end{align*}
The homogeneous Sobolev inequality in $\mathbb R^3$ then gives
\begin{equation}\label{eq:slab-Sobolev-detailed}
 \|f\|_{L^6(S_k)}^2
 \le C\|\nabla f\|_{L^2(S_k)}^2
 +Cr_k^{-2}\|f\|_{L^2(S_k)}^2.
\end{equation}
For $2\le q_1\le6$, put
\[
 \theta:=3\left(\frac12-\frac1{q_1}\right)=\frac2{p_1}.
\]
Interpolation between $L^2$ and \eqref{eq:slab-Sobolev-detailed} yields, for almost every $t$,
\[
 \|u(t)\|_{L^{q_1}(S_k)}
 \le C\|u(t)\|_{L^2(S_k)}^{1-\theta}
 \bigl(\|\nabla u(t)\|_{L^2(S_k)}
 +r_k^{-1}\|u(t)\|_{L^2(S_k)}\bigr)^\theta.
\]
When $p_1<\infty$, raising this inequality to the power $p_1$, integrating over $I_k$, and using $\theta p_1=2$ and $|I_k|\le2r_k^2$ gives
\begin{align*}
 \|u\|_{L^{p_1}(I_k;L^{q_1}(S_k))}^2
 &\le C\left(\esssup_{I_k}\|u(t)\|_{L^2(S_k)}^2\right)^{1-\theta}\\
 &\quad\times\left(\int_{I_k}\|\nabla u(t)\|_{L^2(S_k)}^2dt
 +r_k^{-2}\int_{I_k}\|u(t)\|_{L^2(S_k)}^2dt\right)^\theta\\
 &\le CE_k.
\end{align*}
The case $p_1=\infty$, $q_1=2$ is immediate.  If $k\ge c_0$, the inclusions $I_k\subset I_{k-c_0}$ and $S_k\subset S_{k-c_0}$ give \eqref{eq:endpoint-slab-interpolation} with $C_{c_0}E_{k-c_0}$.  The finitely many indices $0\le k<c_0$ are handled by the whole-space energy interpolation and the convention $E_{k-c_0}=E_{\mathrm g}$.
\end{revision}

For \eqref{eq:endpoint-product-zero}, take $q_1=2m'$ in \eqref{eq:endpoint-slab-interpolation}.  The corresponding time exponent is $p_1=4m/3$.  If $k\ge c_0$, apply \eqref{eq:endpoint-slab-interpolation} at the scale $k-c_0$.  If $0\le k<c_0$, use instead the whole-space energy interpolation on $(-2,0)\times\R^3$; there are only finitely many such $k$, and $r_k$ is bounded below by a positive constant depending on $c_0$.  In either case, the relevant time interval has length bounded by a fixed multiple of $r_k^2$, and restriction of the time norm gives
\begin{align*}
 \norm[L^{2p'}(I_{k-c_0};L^{2m'}(S_{k-c_0}))]{u}^2
 &\le c r_k^{2/p'-4/p_1}E_{k-c_0}\\
 &=c r_k^{2/p'-3/m}E_{k-c_0}
 =c r_kE_{k-c_0},
\end{align*}
where the last equality follows from $2/p+3/m=1$.  Hence
\[
 \norm[L^{p'}(-2,0;L^{m'}(\R^3))]{\chi_k u\otimes u}
 \le \norm[L^{2p'}(I_{k-c_0};L^{2m'}(S_{k-c_0}))]{u}^2
 \le cr_kE_{k-c_0},
\]
which proves \eqref{eq:endpoint-product-zero}.

For the derivative term, apply \eqref{eq:endpoint-slab-interpolation} with
\[
 p_1=\frac{2m}{3},
 \qquad q_1=\frac{2m}{m-2}.
\]
These exponents satisfy
\[
 \frac2{p_1}+\frac3{q_1}=\frac32,
 \qquad \frac1{p'}=\frac1{p_1}+\frac12,
 \qquad \frac1{m'}=\frac1{q_1}+\frac12.
\]
Consequently, with the whole-space interpretation just described when $k<c_0$,
\[
 \norm[L^{p'}(-2,0;L^{m'}(\R^3))]{\chi_k u\nabla u}
 \le c\norm[L^{p_1}(I_{k-c_0};L^{q_1}(S_{k-c_0}))]{u}
       \norm[L^2(I_{k-c_0};L^2(S_{k-c_0}))]{\nabla u}
 \le cE_{k-c_0}.
\]
If the derivative falls on $\chi_k$, then
\begin{align*}
 \norm[L^{p'}(-2,0;L^{m'}(\R^3))]{(\nabla\chi_k)u\otimes u}
 &\le cr_k^{-1}
 \norm[L^{p'}(I_{k-c_0};L^{m'}(S_{k-c_0}))]{u\otimes u}\\
 &\le cr_k^{-1}\cdot r_kE_{k-c_0}
 \le cE_{k-c_0},
\end{align*}
where the local zero-order product bound follows from the same calculation as above.  This proves \eqref{eq:endpoint-product-one}.  The whole-space assertions follow from the same interpolation between $L^\infty_tL^2_x$ and $L^2_tL^6_x$.
\end{proof}

We next introduce nested vertical--parabolic cutoffs.  Choose even functions $\zeta,\zeta^\sharp\in C_c^\infty(\R)$ satisfying $0\le\zeta,\zeta^\sharp\le1$ and nonincreasing as functions of $|s|$, and nonincreasing functions $\rho,\rho^\sharp\in C_c^\infty([0,\infty))$ satisfying $0\le\rho,\rho^\sharp\le1$, such that
\begin{align*}
 &\zeta=1\text{ on }[-1,1],
 &&\operatorname{supp}\zeta\subset[-2,2],
 &\rho=1\text{ on }[0,1],
 &&\operatorname{supp}\rho\subset[0,2),\\
 &\zeta^\sharp=1\text{ on }[-4,4],
 &&\operatorname{supp}\zeta^\sharp\subset[-8,8],
 &\rho^\sharp=1\text{ on }[0,16],
 &&\operatorname{supp}\rho^\sharp\subset[0,32).
\end{align*}
All cutoffs are chosen monotone under dilation.  Put
\begin{equation}\label{eq:endpoint-nested-cutoffs}
 \eta_k(x_3,t):=\zeta(x_3/r_k)\rho((-t)/r_k^2),
 \qquad
 \xi_k(x_3,t):=\zeta^\sharp(x_3/r_k)\rho^\sharp((-t)/r_k^2),
 \qquad \eta:=\eta_0.
\end{equation}
For a fixed $n\ge1$, define
\begin{equation}\label{eq:endpoint-observation-partition}
 \omega_{n,k}:=
 \begin{cases}
  \eta_k-\eta_{k+1},&0\le k<n,\\
  \eta_n,&k=n,
 \end{cases}
 \qquad
 \sum_{k=0}^n\omega_{n,k}=\eta,
\end{equation}
and
\begin{equation}\label{eq:endpoint-source-partition}
 \vartheta_{n,j}:=
 \begin{cases}
  1-\xi_1,&j=0,\\
  \xi_j-\xi_{j+1},&1\le j<n,\\
  \xi_n,&j=n.
 \end{cases}
\end{equation}
Then $\sum_{j=0}^n\vartheta_{n,j}=1$ on $\R^3\times(-2,0)$.  With $\Phi_n$ from \eqref{eq:Phi}, set
\begin{equation}\label{eq:endpoint-Wnk}
 W_{n,k}:=\omega_{n,k}\partial_3\Phi_n.
\end{equation}

\begin{lemma}[Endpoint cutoff geometry]\label{lem:endpoint-cutoff-geometry}
There is an absolute integer $L\ge8$ such that, for $0\le k\le n$,
\begin{equation}\label{eq:endpoint-W-bounds}
 \operatorname{supp}W_{n,k}\subset\mathcal Q_{k-2},
 \qquad
 \norm[L^\infty((-2,0)\times\R^3)]{W_{n,k}}\le c r_k^{-2},
 \qquad
 \norm[L^\infty((-2,0)\times\R^3)]{\nabla W_{n,k}}\le c r_k^{-3}.
\end{equation}
If $k\ge L+2$, $0\le j\le k-L-1$, and $t\in I_{k-2}$, then
\begin{equation}\label{eq:endpoint-separation}
 \operatorname{dist}\bigl(\operatorname{supp}\vartheta_{n,j}(\cdot,t),S_{k-2}\bigr)
 \ge c r_j.
\end{equation}
Finally, if $J=k-L\ge1$, then
\begin{equation}\label{eq:endpoint-source-telescope}
 \sum_{j=J}^{n}\vartheta_{n,j}=\xi_J.
\end{equation}
\end{lemma}

\begin{proof}
Write
\[
 \tau:=-t+r_n^2.
\]
Direct differentiation of the Gaussian shows that, for $a=1,2$,
\begin{equation*}
 |\partial_3^a\Phi_n(x_3,t)|
 \le c_a\tau^{-(a+1)/2}
 \left(1+\frac{|x_3|}{\sqrt\tau}\right)^{a+1}
 \exp\left(-\frac{x_3^2}{8\tau}\right).
\end{equation*}

\emph{Step 1: support of the observation pieces.}
Since $\operatorname{supp}\omega_{n,k}\subset\operatorname{supp}\eta_k$, the definitions of $\zeta$ and $\rho$ give
\[
 |x_3|<2r_k,
 \qquad -t<2r_k^2
 \quad\text{on }\operatorname{supp}\omega_{n,k}.
\]
Because $2r_k<2r_{k-2}$ and $2r_k^2<2r_{k-2}^2$, this implies
\[
 \operatorname{supp}W_{n,k}\subset\mathcal Q_{k-2}.
\]

\emph{Step 2: bounds for $W_{n,k}$ and its spatial gradient.}
If $k=n$, then $\tau\ge r_n^2=r_k^2$, and the Gaussian derivative bound above immediately yields
\[
 |\partial_3\Phi_n|\le cr_k^{-2},
 \qquad
 |\partial_3^2\Phi_n|\le cr_k^{-3}
 \quad\text{on }\operatorname{supp}\omega_{n,n}.
\]
Assume now that $k<n$.  The cutoff $\eta_{k+1}$ equals one whenever
\[
 |x_3|\le r_{k+1},
 \qquad -t\le r_{k+1}^2.
\]
Hence every point in $\operatorname{supp}(\eta_k-\eta_{k+1})$ satisfies at least one of
\begin{equation*}
 -t\ge r_{k+1}^2
 \qquad\text{or}\qquad
 |x_3|\ge r_{k+1}.
\end{equation*}
In the first case, $\tau\ge r_{k+1}^2=r_k^2/4$, so the desired estimates follow directly from the Gaussian derivative bound above.  In the second case, set $y=|x_3|/\sqrt\tau$.  Since $|x_3|\ge r_{k+1}=r_k/2$,
\[
 \tau^{-(a+1)/2}(1+y)^{a+1}e^{-y^2/8}
 =|x_3|^{-(a+1)}y^{a+1}(1+y)^{a+1}e^{-y^2/8}
 \le cr_k^{-(a+1)}
\]
for $a=1,2$.  Thus the same derivative bounds hold in both alternatives in the preceding dichotomy.

Since $|\partial_3\omega_{n,k}|\le cr_k^{-1}$ and the horizontal derivatives of $\omega_{n,k}$ and $\Phi_n$ vanish,
\begin{align*}
 |W_{n,k}|
 &\le cr_k^{-2},\\
 |\nabla W_{n,k}|
 &=|\partial_3W_{n,k}|\\
 &\le |\partial_3\omega_{n,k}|\,|\partial_3\Phi_n|
      +|\omega_{n,k}|\,|\partial_3^2\Phi_n|
 \le cr_k^{-3}.
\end{align*}
This proves \eqref{eq:endpoint-W-bounds}.

\emph{Step 3: separation of a remote source from the observation slab.}
Let $k\ge L+2$, $0\le j\le k-L-1$, and $t\in I_{k-2}$.  Then
\[
 -t\le2r_{k-2}^2=32r_k^2.
\]
Since $r_k/r_{j+1}\le2^{-L}$, choosing $L$ sufficiently large gives
\[
 -t\le16r_{j+1}^2.
\]
Therefore the temporal factors in $\xi_j$ and $\xi_{j+1}$ are equal to one.  If $j\ge1$, then $\xi_{j+1}=1$ for $|x_3|\le4r_{j+1}=2r_j$, and hence
\[
 \operatorname{supp}(\xi_j-\xi_{j+1})(\cdot,t)
 \subset\{|x_3|\ge2r_j\}.
\]
For $j=0$, the same conclusion, with another fixed positive multiple of $r_0$, follows because $1-\xi_1$ vanishes on the central slab where $\xi_1=1$.  On the other hand, the vertical half-width of $S_{k-2}$ is $8r_k$.  Enlarging $L$ once more, if necessary, gives
\[
 \operatorname{dist}\bigl(\operatorname{supp}\vartheta_{n,j}(\cdot,t),S_{k-2}\bigr)
 \ge cr_j,
\]
which is \eqref{eq:endpoint-separation}.

\emph{Step 4: telescoping of the nonseparated source pieces.}
If $J=k-L\ge1$, then
\[
 (\xi_J-\xi_{J+1})+\cdots+(\xi_{n-1}-\xi_n)+\xi_n=\xi_J.
\]
By the definition of $\vartheta_{n,j}$ this is precisely \eqref{eq:endpoint-source-telescope}.
\end{proof}

Fix once and for all an integer $c_*\ge L+4$ large enough to absorb every fixed cutoff enlargement used below.

Define the two-sided kernel
\begin{equation}\label{eq:endpoint-kernel}
 \mathcal K_d:=
 \begin{cases}
  2^{-d/m},&d\ge L,\\
  1,&|d|<L,\\
  2^d,&d\le-L,
 \end{cases}
\end{equation}
and, for $k\ge0$,
\begin{equation}\label{eq:endpoint-Bk}
 B_k:=\sum_{j\in\Z}\mathcal K_{k-j}a_j.
\end{equation}
Choose $m<p_2<\infty$ and let $q_2>1$ be determined by
\begin{equation}\label{eq:endpoint-p2-q2}
 \frac2{p_2}+\frac3{q_2}=3.
\end{equation}
Since $p_2>m$, we have
\[
 \frac1{q_2}=1-\frac{2}{3p_2}>1-\frac1m=\frac1{m'},
\]
and hence $q_2<m'$.  Set
\begin{equation}\label{eq:endpoint-gamma-envelope}
 \gamma:=\frac2m-\frac2{p_2}>0,
 \qquad
 \widehat B_j:=\sum_{k\ge j}2^{-\gamma(k-j)}B_k.
\end{equation}
Since $\mathcal K\in\ell^1(\Z)$, discrete Young's inequality and Tonelli's theorem give
\begin{equation}\label{eq:endpoint-summability}
 \sum_{k\ge0}B_k+\sum_{j\ge0}\widehat B_j
 \le c_{p,p_2}\sum_{j\in\Z}a_j,
\end{equation}
and both tails tend to zero.

\begin{lemma}[Frequency-resolved endpoint flux]\label{lem:endpoint-flux}
Fix $0\le k\le n$.  Let $H_k$ be a scalar function on $\mathcal Q_{k-2}$ satisfying
\begin{equation}\label{eq:endpoint-H-bounds}
 \norm[L^{p'}(I_{k-2};L^{m'}(S_{k-2}))]{H_k}\le M_k,
 \qquad
 \norm[L^{p'}(I_{k-2};L^{m'}(S_{k-2}))]{\nabla H_k}
 \le c r_k^{-1}M_k.
\end{equation}
Then, for every $\tau\in I_n$,
\begin{equation}\label{eq:endpoint-flux-bound}
 \abs{\int_{-2}^{\tau}\int_{\R^3}u^3H_kW_{n,k}\dd x\dd t}
 \le c B_k r_k^{-2}M_k.
\end{equation}
The constant is independent of $\tau$.
\end{lemma}

\begin{proof}
Fix $\tau\in I_n$.  All space--time integrals in this proof are over $(-2,\tau)\times\R^3$.  We first replace $u^3$ by a finite Littlewood--Paley sum; all estimates below are independent of the truncation.

For $j\le k-L$, the support property $\operatorname{supp}W_{n,k}\subset\mathcal Q_{k-2}$, the vertical thickness of $S_{k-2}$, and the one-dimensional Bernstein inequality give
\begin{equation}\label{eq:endpoint-low-frequency-gain}
 \norm[L^p(I_{k-2};L^m(S_{k-2}))]{\DeltaDot_j u^3}
 \le c r_k^{1/m}2^{j/m}a_j
 \le c2^{-(k-j)/m}a_j.
\end{equation}
Consequently, H\"older's inequality, \eqref{eq:endpoint-H-bounds}, and
$\|W_{n,k}\|_{L^\infty((-2,0)\times\R^3)}\le cr_k^{-2}$ imply
\[
 \abs{\int \DeltaDot_j u^3H_kW_{n,k}}
 \le c2^{-(k-j)/m}a_jr_k^{-2}M_k.
\]
If $|j-k|<L$, the same argument uses the global definition of $a_j$ and gives
\[
 \abs{\int \DeltaDot_j u^3H_kW_{n,k}}
 \le ca_jr_k^{-2}M_k.
\]

It remains to treat $j\ge k+L$.
\begin{revision}
To avoid an unjustified zero extension of $H_k$, choose a vertical cutoff
\[
 \chi_k^{\mathrm{buf}}\in C_c^\infty((-2r_{k-2},2r_{k-2}))
\]
such that $\chi_k^{\mathrm{buf}}=1$ on the vertical projection of $\operatorname{supp}W_{n,k}$ and
\[
 |\partial_3\chi_k^{\mathrm{buf}}|\le Cr_k^{-1}.
\]
For the fixed terminal time $\tau$, define
\[
 F_{k,\tau}(x,t):=
 \mathbf1_{(-2,\tau)}(t)\,\chi_k^{\mathrm{buf}}(x_3)H_k(x,t)W_{n,k}(x_3,t)
\]
for $(x,t)\in S_{k-2}\times I_{k-2}$ and extend it by zero in the spatial variable outside $S_{k-2}$.  Because $\chi_k^{\mathrm{buf}}$ vanishes near the vertical boundary, this is a legitimate $W^{1,m'}(\mathbb R^3)$ extension for almost every time.  Moreover, $\chi_k^{\mathrm{buf}}=1$ on $\operatorname{supp}W_{n,k}$, so the original pairing is unchanged.  From \eqref{eq:endpoint-W-bounds} and \eqref{eq:endpoint-H-bounds},
\begin{align*}
 \|F_{k,\tau}\|_{L^{p'}(-2,0;L^{m'}(\mathbb R^3))}
 &\le Cr_k^{-2}M_k,\\
 \|\nabla F_{k,\tau}\|_{L^{p'}(-2,0;L^{m'}(\mathbb R^3))}
 &\le Cr_k^{-3}M_k.
\end{align*}
Indeed, the spatial gradient is the sum of
\[
 \chi_k^{\mathrm{buf}}(\nabla H_k)W_{n,k},\qquad
 \chi_k^{\mathrm{buf}} H_k\nabla W_{n,k},\qquad
 (\nabla\chi_k^{\mathrm{buf}})H_kW_{n,k},
\]
and each term has the stated bound.  The time indicator creates no spatial boundary term.
\end{revision}
Let $\widetilde\Delta_j$ be a real, even Fourier multiplier equal to one on the Fourier support of $\DeltaDot_j$, with symbol supported in an annulus comparable to $2^j$.  Since the symbol vanishes near the origin, the inverse Bernstein estimate gives, for almost every time,
\begin{equation*}
 \norm[L^{m'}(\R^3)]{\widetilde\Delta_jF_{k,\tau}(t)}
 \le c2^{-j}\norm[L^{m'}(\R^3)]{\nabla F_{k,\tau}(t)}.
\end{equation*}
Moreover, $\widetilde\Delta_j\DeltaDot_j=\DeltaDot_j$, and the real even symbol makes $\widetilde\Delta_j$ self-adjoint.  Hence
\[
 \int\DeltaDot_j u^3H_kW_{n,k}
 =\int\DeltaDot_j u^3F_{k,\tau}
 =\int\DeltaDot_j u^3\,\widetilde\Delta_jF_{k,\tau}.
\]
Using H\"older's inequality in time and space, the inverse Bernstein estimate above, and the gradient estimate for $F_{k,\tau}$ above, we obtain
\begin{align*}
 \abs{\int\DeltaDot_j u^3H_kW_{n,k}}
 &\le a_j
 \norm[L^{p'}(-2,0;L^{m'}(\R^3))]{\widetilde\Delta_jF_{k,\tau}}\\
 &\le ca_j2^{-j}r_k^{-3}M_k
 =c2^{-(j-k)}a_jr_k^{-2}M_k.
\end{align*}
Summing the low, neighboring, and high frequency ranges produces the kernel $\mathcal K_{k-j}$ in \eqref{eq:endpoint-kernel}, and hence gives \eqref{eq:endpoint-flux-bound}.  Finally, Lemma~\ref{lem:homogeneous-reconstruction} permits passage from the finite dyadic sums to the homogeneous reconstruction because $\sum_ja_j<\infty$.
\end{proof}

\begin{revision}
For each $k$, choose a cutoff $\chi_k$ as in Lemma~\ref{lem:endpoint-products}, with the fixed shift $c_*$, such that $\chi_k=1$ on a neighborhood of $\operatorname{supp}W_{n,k}$, and set
\[
 H_k:=\chi_k|u|^2.
\]
Applying Lemma~\ref{lem:endpoint-products} componentwise to the diagonal entries of $\chi_k u\otimes u$ gives
\[
 \norm[L^{p'}(-2,0;L^{m'}(\mathbb R^3))]{H_k}
 \le cr_kE_{k-c_*},
 \qquad
 \norm[L^{p'}(-2,0;L^{m'}(\mathbb R^3))]{\nabla H_k}
 \le cE_{k-c_*}.
\]
In particular, the same estimates hold after restriction to $I_{k-2}\times S_{k-2}$.  Since $\chi_k=1$ on $\operatorname{supp}W_{n,k}$,
\[
 H_kW_{n,k}=|u|^2W_{n,k}.
\]
Thus Lemma~\ref{lem:endpoint-flux} applies with $M_k=cr_kE_{k-c_*}$.  Since $\eta\partial_3\Phi_n=\sum_{k=0}^nW_{n,k}$, we obtain
\end{revision}
\begin{equation}\label{eq:endpoint-convection-flux}
 \sup_{\tau\in I_n}
 \abs{\int_{-2}^{\tau}\int_{\R^3}|u|^2u^3\eta\partial_3\Phi_n\dd x\dd t}
 \le c\sum_{k=0}^{n}B_ke_{k-c_*}.
\end{equation}

The pressure estimate uses harmonicity between the source and observation scales.

For $\rho>0$, write
\[
 \mathcal S_\rho:=\R^2\times(-\rho,\rho).
\]

\begin{lemma}[Harmonic observation on a thin slab]\label{lem:endpoint-harmonic-slab}
Let $0<\rho\le r/16$, and let $h\in L^{q_2}(\R^3)$ be harmonic in the open $r$-neighborhood of $\mathcal S_{4\rho}$.  For $a=0,1$,
\begin{equation}\label{eq:endpoint-harmonic-slab}
 \norm[L^{m'}(\mathcal S_{4\rho})]{\nabla^a h}
 \le c\rho^{1/m'}r^{2/m'-a-3/q_2}
 \norm[L^{q_2}(\R^3)]{h}.
\end{equation}
The constant depends only on $m'$, $q_2$, and the fixed relative size of the harmonicity neighborhood.
\end{lemma}

\begin{proof}
Choose fixed constants $0<c_1<c_2<1/2$ so that, whenever $\rho\le r/16$, a horizontal disk of radius $c_1r$ together with the vertical interval $(-4\rho,4\rho)$ is contained in a three-dimensional ball of radius $c_2r$.  Cover $\R^2$ by disks
\[
 D_\alpha:=B_h(x_\alpha',c_1r)
\]
so that every fixed enlargement of the covering has bounded overlap, and set
\[
 B_\alpha:=B((x_\alpha',0),c_2r).
\]
The constants $c_1,c_2$ may be chosen so that
\begin{equation*}
 D_\alpha\times(-4\rho,4\rho)\subset B_\alpha,
 \qquad
 2B_\alpha\subset
 \{x:\operatorname{dist}(x,\mathcal S_{4\rho})<r\}.
\end{equation*}
Thus $h$ is harmonic in $2B_\alpha$ for every $\alpha$.

For $a=0,1$, the interior $L^{q_2}$ estimate for harmonic functions, applied from $2B_\alpha$ to $B_\alpha$, gives
\begin{equation*}
 \norm[L^\infty(D_\alpha\times(-4\rho,4\rho))]{\nabla^a h}
 \le cr^{-a-3/q_2}\norm[L^{q_2}(2B_\alpha)]{h}.
\end{equation*}
Since each box $D_\alpha\times(-4\rho,4\rho)$ has volume at most $cr^2\rho$, the covering and the local harmonic estimate above imply
\begin{align*}
 \norm[L^{m'}(\mathcal S_{4\rho})]{\nabla^a h}
 &\le c\rho^{1/m'}r^{2/m'-a-3/q_2}
 \left(\sum_\alpha
 \norm[L^{q_2}(2B_\alpha)]{h}^{m'}\right)^{1/m'}.
\end{align*}
Because $q_2\le m'$, the embedding $\ell^{q_2}\hookrightarrow\ell^{m'}$ gives
\begin{align*}
 \left(\sum_\alpha
 \norm[L^{q_2}(2B_\alpha)]{h}^{m'}\right)^{1/m'}
 &\le
 \left(\sum_\alpha
 \norm[L^{q_2}(2B_\alpha)]{h}^{q_2}\right)^{1/q_2}\\
 &=\left(\sum_\alpha\int_{2B_\alpha}|h|^{q_2}\dd x\right)^{1/q_2}\\
 &\le c\norm[L^{q_2}(\R^3)]{h},
\end{align*}
where the last step uses the bounded overlap of the enlarged balls.  Substitution into the preceding estimate proves \eqref{eq:endpoint-harmonic-slab} for $a=0$ and $a=1$.
\end{proof}

For $0\le j\le n$, define
\begin{equation}\label{eq:endpoint-pressure-pieces}
 F_{n,j}^{ab}:=u^au^b\vartheta_{n,j},
 \qquad
 P_{n,j}:=\sum_{a,b=1}^3\mathcal R_a\mathcal R_b(F_{n,j}^{ab}).
\end{equation}
By \eqref{eq:endpoint-pressure-normalization} and \eqref{eq:endpoint-source-partition},
\begin{equation}\label{eq:endpoint-pressure-partition}
 \pi=\sum_{j=0}^{n}P_{n,j}.
\end{equation}

\begin{lemma}[Separated endpoint pressure source]\label{lem:endpoint-separated-pressure}
Assume $k\ge L+2$ and $0\le j\le k-L-1$.  Put
\begin{equation}\label{eq:endpoint-Ajk}
 A_{j,k}:=
 r_k^{1/m'}r_j^{2/m'-3/q_2}
 r_k^{2/p'-2/p_2}E_{j-c_*}.
\end{equation}
Then
\begin{align}
 \norm[L^{p'}(I_{k-2};L^{m'}(S_{k-2}))]{P_{n,j}}
 &\le cA_{j,k},
 \label{eq:endpoint-separated-zero}\\
 \norm[L^{p'}(I_{k-2};L^{m'}(S_{k-2}))]{\nabla P_{n,j}}
 &\le cr_j^{-1}A_{j,k}
 \le cr_k^{-1}A_{j,k}.
 \label{eq:endpoint-separated-one}
\end{align}
Moreover,
\begin{equation}\label{eq:endpoint-separated-scale}
 r_k^{-2}A_{j,k}
 =r_j^{-1}\left(\frac{r_k}{r_j}\right)^\gamma E_{j-c_*}
 \le c2^{-\gamma(k-j)}e_{j-c_*}.
\end{equation}
\end{lemma}

\begin{proof}
By Lemma~\ref{lem:endpoint-cutoff-geometry}, the source of $P_{n,j}(\cdot,t)$ is separated from $S_{k-2}$ by a distance comparable to $r_j$.  Hence $P_{n,j}(\cdot,t)$ is harmonic in a fixed $r_j$-neighborhood of the observation slab.

We first estimate the source.  Since
\[
 \frac2{2p_2}+\frac3{2q_2}=\frac32,
\]
the pair $(2p_2,2q_2)$ lies on the energy line.  For $j\ge c_*$, the support of $\vartheta_{n,j}$ is contained in a fixed enlargement of $\mathcal Q_j$, which is contained in $\mathcal Q_{j-c_*}$ by the choice of $c_*$.  Therefore $|\vartheta_{n,j}|\le1$ and Lemma~\ref{lem:endpoint-products} give
\begin{align*}
 \norm[L^{p_2}(I_{j-c_*};L^{q_2}(\R^3))]{F_{n,j}}
 &\le c\norm[L^{2p_2}(I_{j-c_*};L^{2q_2}(S_{j-c_*}))]{u}^2\\
 &\le cE_{j-c_*}.
\end{align*}
For the finitely many indices $0\le j<c_*$, the same estimate holds on $(-2,0)$ with the right-hand side $cE_{\mathrm g}=cE_{j-c_*}$, by whole-space energy interpolation.  In either case the relevant time interval contains $I_{k-2}$.  Since $p_2>p'$ and $|I_{k-2}|\le cr_k^2$,
\[
 \norm[L^{p'}(I_{k-2};L^{q_2}(\R^3))]{F_{n,j}}
 \le cr_k^{2/p'-2/p_2}E_{j-c_*}.
\]
Calder\'on--Zygmund boundedness transfers this estimate to $P_{n,j}$ in $L^{q_2}(\R^3)$.  Choose $c_0>0$ from the separation estimate \eqref{eq:endpoint-separation}.  Lemma~\ref{lem:endpoint-harmonic-slab}, with $\rho=2r_k$ and $r=c_0r_j/2$, applies once $L$ is fixed sufficiently large; indeed, $\mathcal S_{4\rho}=S_{k-2}$ and $\rho\le r/16$.  It proves \eqref{eq:endpoint-separated-zero}--\eqref{eq:endpoint-separated-one}.

The exponent identities
\[
 -2+\frac1{m'}+\frac2{p'}-\frac2{p_2}=\gamma,
 \qquad
 \frac2{m'}-\frac3{q_2}=-1-\gamma
\]
give the equality in \eqref{eq:endpoint-separated-scale}.  Its final inequality follows from \eqref{eq:endpoint-negative-index-bound} and $r_k/r_j=2^{-(k-j)}$.
\end{proof}

For each observation scale, set
\begin{equation}\label{eq:endpoint-near-pressure}
 \Pi_k:=
 \begin{cases}
  \pi,&0\le k\le L+1,\\
  \displaystyle\sum_{a,b=1}^3\mathcal R_a\mathcal R_b
  (u^au^b\xi_{k-L}),&k\ge L+2.
 \end{cases}
\end{equation}
The telescoping identity \eqref{eq:endpoint-source-telescope} gives, for $k\ge L+2$,
\begin{equation}\label{eq:endpoint-observation-pressure-decomposition}
 \pi=\Pi_k+\sum_{j=0}^{k-L-1}P_{n,j}
\end{equation}
on the observation scale $k$.

\begin{lemma}[Grouped endpoint pressure]\label{lem:endpoint-grouped-pressure}
For every $0\le k\le n$,
\begin{align}
 \norm[L^{p'}(I_{k-2};L^{m'}(S_{k-2}))]{\Pi_k}
 &\le cr_kE_{k-c_*},
 \label{eq:endpoint-near-pressure-zero}\\
 \norm[L^{p'}(I_{k-2};L^{m'}(S_{k-2}))]{\nabla\Pi_k}
 &\le cE_{k-c_*}.
 \label{eq:endpoint-near-pressure-one}
\end{align}
\end{lemma}

\begin{proof}
We distinguish the finitely many outer observation scales from the genuinely localized scales.

Suppose first that $0\le k\le L+1$.  Then $\Pi_k=\pi$, and $k-c_*<0$ because $c_*\ge L+4$.  By the pressure normalization \eqref{eq:endpoint-pressure-normalization}, Calder\'on--Zygmund boundedness, and the whole-space form of Lemma~\ref{lem:endpoint-products},
\begin{align*}
 \norm[L^{p'}(I_{k-2};L^{m'}(S_{k-2}))]{\Pi_k}
 &\le \norm[L^{p'}(I_{k-2};L^{m'}(\R^3))]{\pi}\\
 &\le c\norm[L^{p'}(I_{k-2};L^{m'}(\R^3))]{u\otimes u}
 \le c_LE_{\mathrm g}.
\end{align*}
Since $r_k\ge2^{-(L+1)}$ and $E_{k-c_*}=E_{\mathrm g}$, the last term is bounded by $c r_kE_{k-c_*}$ after changing the fixed constant.  Similarly,
\begin{align*}
 \norm[L^{p'}(I_{k-2};L^{m'}(S_{k-2}))]{\nabla\Pi_k}
 &\le c\norm[L^{p'}(I_{k-2};L^{m'}(\R^3))]{\nabla(u\otimes u)}
 \le cE_{\mathrm g}=cE_{k-c_*}.
\end{align*}
This proves both estimates for $k\le L+1$.

Let now $k\ge L+2$.  In this case
\[
 \Pi_k=\sum_{a,b=1}^3\mathcal R_a\mathcal R_b
 (u^au^b\xi_{k-L}).
\]
The cutoff $\xi_{k-L}$ is supported in a fixed enlargement of the slab at scale $r_{k-L}$ and satisfies
\[
 |\partial_3\xi_{k-L}|\le cr_{k-L}^{-1}.
\]
Apply Lemma~\ref{lem:endpoint-products} at the scale $k-L$ with the fixed shift $c_*-L$.  Since $c_*\ge L+4$, this shift is nonnegative and
\[
 (k-L)-(c_*-L)=k-c_*.
\]
Hence
\begin{align*}
 \norm[L^{p'}(-2,0;L^{m'}(\R^3))]{u\otimes u\,\xi_{k-L}}
 &\le cr_{k-L}E_{k-c_*},\\
 \norm[L^{p'}(-2,0;L^{m'}(\R^3))]{\nabla(u\otimes u\,\xi_{k-L})}
 &\le cE_{k-c_*}.
\end{align*}
Riesz transforms are bounded on $L^{m'}(\R^3)$ and commute with spatial derivatives.  Restricting the resulting pressure estimates to $I_{k-2}\times S_{k-2}$ and using $r_{k-L}=2^Lr_k$, with $L$ fixed, gives
\begin{align*}
 \norm[L^{p'}(I_{k-2};L^{m'}(S_{k-2}))]{\Pi_k}
 &\le cr_kE_{k-c_*},\\
 \norm[L^{p'}(I_{k-2};L^{m'}(S_{k-2}))]{\nabla\Pi_k}
 &\le cE_{k-c_*}.
\end{align*}
This is \eqref{eq:endpoint-near-pressure-zero}--\eqref{eq:endpoint-near-pressure-one}.
\end{proof}

\begin{proposition}[Endpoint pressure flux]\label{prop:endpoint-pressure-flux}
For every $n\ge1$,
\begin{equation}\label{eq:endpoint-pressure-flux}
 \sup_{\tau\in I_n}
 \abs{\int_{-2}^{\tau}\int_{\R^3}\pi u^3\eta\partial_3\Phi_n\dd x\dd t}
 \le C+c\sum_{j=0}^{n}
 \bigl(B_{j+c_*}+\widehat B_{j+c_*}\bigr)e_j,
\end{equation}
where
\[
 C=C\!\left(p,m,p_2,E_{\mathrm g},
 \|u^3\|_{\widetilde L^p(-2,0;\dot B^0_{m,1})}\right)
\]
is independent of $n$ and $\tau$; the fixed cutoffs are suppressed from the notation.  The coefficient sequence on the right belongs to $\ell^1$ and has vanishing tails.
\end{proposition}

\begin{proof}
Fix $\tau\in I_n$.  All integrals without displayed limits in this proof are over $(-2,\tau)\times\R^3$.  Since
\[
 \eta\partial_3\Phi_n=\sum_{k=0}^nW_{n,k},
\]
we decompose the pressure flux according to the observation scale.  For $0\le k\le L+1$, the definition gives $\Pi_k=\pi$.  For $k\ge L+2$, the exact source decomposition \eqref{eq:endpoint-observation-pressure-decomposition} yields
\[
 \pi=\Pi_k+\sum_{j=0}^{k-L-1}P_{n,j}
 \quad\text{on }\operatorname{supp}W_{n,k}.
\]
Consequently,
\begin{align*}
 \int\pi u^3\eta\partial_3\Phi_n
 &=\sum_{k=0}^n\int u^3\Pi_kW_{n,k}\\
 &\quad+\sum_{k=L+2}^n\sum_{j=0}^{k-L-1}
 \int u^3P_{n,j}W_{n,k}.
\end{align*}

For the grouped pressure, Lemma~\ref{lem:endpoint-grouped-pressure} verifies the hypotheses of Lemma~\ref{lem:endpoint-flux} with
\[
 H_k=\Pi_k,
 \qquad M_k=cr_kE_{k-c_*}.
\]
Therefore, by \eqref{eq:endpoint-negative-index-bound},
\begin{align*}
 \abs{\int u^3\Pi_kW_{n,k}}
 &\le cB_kr_k^{-2}(r_kE_{k-c_*})\\
 &\le cB_ke_{k-c_*}.
\end{align*}

For a separated source $0\le j\le k-L-1$, Lemma~\ref{lem:endpoint-separated-pressure} verifies the same flux-pairing hypotheses with
\[
 H_k=P_{n,j},
 \qquad M_k=cA_{j,k}.
\]
Using \eqref{eq:endpoint-separated-scale}, we obtain
\begin{align*}
 \abs{\int u^3P_{n,j}W_{n,k}}
 &\le cB_kr_k^{-2}A_{j,k}\\
 &\le c2^{-\gamma(k-j)}B_ke_{j-c_*}.
\end{align*}
Inserting the grouped-pressure estimate above and
the separated-source estimate above into
the decomposition above gives
\begin{align*}
 \abs{\int\pi u^3\eta\partial_3\Phi_n}
 &\le c\sum_{k=0}^{n}B_ke_{k-c_*}\\
 &\quad+c\sum_{k=L+2}^{n}\sum_{j=0}^{k-L-1}
 2^{-\gamma(k-j)}B_ke_{j-c_*}.
\end{align*}
Since all terms are nonnegative, the second sum may be enlarged and its order reversed:
\begin{align*}
 \sum_{k=L+2}^{n}\sum_{j=0}^{k-L-1}
 2^{-\gamma(k-j)}B_ke_{j-c_*}
 &\le\sum_{j=0}^{n}e_{j-c_*}
 \sum_{k\ge j}2^{-\gamma(k-j)}B_k\\
 &=\sum_{j=0}^{n}\widehat B_je_{j-c_*}.
\end{align*}
Thus
\begin{equation*}
 \abs{\int\pi u^3\eta\partial_3\Phi_n}
 \le c\sum_{j=0}^{n}(B_j+\widehat B_j)e_{j-c_*}.
\end{equation*}
For $0\le j<c_*$, the negative-index convention gives $e_{j-c_*}=E_{\mathrm g}$.  The corresponding finite sum is bounded by
\[
 cE_{\mathrm g}\sum_{j=0}^{c_*-1}(B_j+\widehat B_j),
\]
which is a constant independent of $n$ and $\tau$.  For $j\ge c_*$, set $i=j-c_*$.  Enlarging the finite upper limit if necessary gives
\[
 \sum_{j=c_*}^{n}(B_j+\widehat B_j)e_{j-c_*}
 \le\sum_{i=0}^{n}(B_{i+c_*}+\widehat B_{i+c_*})e_i.
\]
All bounds above are independent of $\tau\in I_n$.  Taking the supremum in $\tau$ proves \eqref{eq:endpoint-pressure-flux}.  Finally,
\eqref{eq:endpoint-summability} shows that the coefficient sequence belongs to $\ell^1$ and has vanishing tails.
\end{proof}

In related one-component local-energy arguments, the backward heat-kernel test is usually handled inside the main energy estimate; see, for example, \cite[Section~3]{ChaeWolf2021}.  Here the observation slab is unbounded in the horizontal variables, so the approximation by compactly supported tests must be justified separately.  The next two lemmas have distinct roles: Lemma~\ref{lem:endpoint-horizontal-cutoff} establishes admissibility of the full-slab test, whereas Lemma~\ref{lem:endpoint-coercivity} extracts the scale-invariant energy from its left-hand side.

\begin{lemma}[Horizontal cutoff limit]\label{lem:endpoint-horizontal-cutoff}
For every fixed $n$, the local energy inequality may be tested with $\Phi_n\eta$.  More precisely, if $\psi_R(x_h)$ is a standard radial cutoff that equals one on $B_h(0,R)$ and vanishes outside $B_h(0,2R)$, then the local energy inequality with $\Phi_n\eta\psi_R$ passes to the limit as $R\to\infty$.
\end{lemma}

\begin{proof}
Fix $n\ge1$.  Choose a radial function $\psi\in C_c^\infty(\R^2)$, nonincreasing in $|x_h|$, with
\[
 0\le\psi\le1,
 \qquad \psi=1\text{ on }B_h(0,1),
 \qquad \operatorname{supp}\psi\subset B_h(0,2),
\]
and set $\psi_R(x_h):=\psi(x_h/R)$.  Then
\[
 |\nabla_h\psi_R|\le cR^{-1},
 \qquad |\Delta_h\psi_R|\le cR^{-2},
\]
with both derivatives supported in $\{R<|x_h|<2R\}$.  For almost every terminal time $\tau\in I_n$, the function
\[
 \phi_{n,R}(x,t):=\Phi_n(x_3,t)\eta(x_3,t)\psi_R(x_h)
\]
is a nonnegative, compactly supported admissible test function in \eqref{eq:LEI}; its temporal cutoff vanishes at $t=-2$.

Since $\Phi_n\eta$ is independent of $x_h$,
\begin{align*}
 (\partial_t+\Delta)\phi_{n,R}
 &=\psi_R(\partial_t+\partial_3^2)(\Phi_n\eta)
   +\Phi_n\eta\,\Delta_h\psi_R,
 \\
 u\cdot\nabla\phi_{n,R}
 &=\psi_Ru^3\partial_3(\Phi_n\eta)
   +\Phi_n\eta\,u_h\cdot\nabla_h\psi_R.
\end{align*}
For fixed $n$, $\Phi_n\eta$ and the derivatives appearing in
the two identities above are bounded.  Let
\[
 \mathcal A_R:=\{R<|x_h|<2R\}\times\R\times(-2,0).
\]
The term containing $\Delta_h\psi_R$ is bounded by
\[
 c_nR^{-2}\int_{\mathcal A_R}|u|^2\dd x\dd t,
\]
the horizontal convection term by
\[
 c_nR^{-1}\int_{\mathcal A_R}|u|^3\dd x\dd t,
\]
and the horizontal pressure term by
\[
 c_nR^{-1}\int_{\mathcal A_R}|\pi||u|\dd x\dd t.
\]
At the beginning of this section we established
\[
 |u|^2,\ |u|^3,\ |\pi||u|\in L^1(\R^3\times(-2,0)).
\]
Hence all three annular integrals tend to zero as $R\to\infty$ by absolute continuity.  The remaining right-hand terms converge by dominated convergence, because $0\le\psi_R\le1$ and $\psi_R\to1$ pointwise.

Choose the radial cutoffs so that $\psi_R$ increases with $R$.  For almost every terminal time $\tau$, monotone convergence gives
\[
 \int_{\R^3}|u(x,\tau)|^2\Phi_n\eta\psi_R\dd x
 \longrightarrow
 \int_{\R^3}|u(x,\tau)|^2\Phi_n\eta\dd x.
\]
The same argument, or Fatou's lemma followed by monotone convergence, applies to the nonnegative dissipation term.  Letting $R\to\infty$ in the local energy inequality therefore yields exactly the inequality obtained by testing with $\Phi_n\eta$.
\end{proof}

\begin{lemma}[Coercivity of the endpoint test]\label{lem:endpoint-coercivity}
There is $c>0$, independent of $n$, such that
\begin{equation}\label{eq:endpoint-Phi-lower}
 \Phi_n(x_3,t)\eta(x_3,t)\ge cr_n^{-1}
 \quad\text{on }\mathcal Q_n,
 \qquad n\ge1.
\end{equation}
Consequently, the left-hand side of the local energy inequality tested with $\Phi_n\eta$ controls $ce_n$.
\end{lemma}

\begin{proof}
Let $(x,t)\in\mathcal Q_n$ with $n\ge1$.  Then
\[
 |x_3|\le2r_n,
 \qquad -2r_n^2<t<0.
\]
Since $2r_n\le1$ and $2r_n^2\le1$, the definition of the outer cutoff gives $\eta(x_3,t)=1$ on $\mathcal Q_n$.  Moreover,
\[
 r_n^2\le -t+r_n^2\le3r_n^2,
 \qquad
 \frac{x_3^2}{4(-t+r_n^2)}\le1.
\]
Therefore
\[
 \Phi_n(x_3,t)
 \ge \frac{e^{-1}}{\sqrt{12\pi}}\,r_n^{-1}
 =:c_0r_n^{-1},
\]
which proves \eqref{eq:endpoint-Phi-lower}.

We now explain the consequence for the local energy inequality.  Apply the inequality, justified by Lemma~\ref{lem:endpoint-horizontal-cutoff}, from $t=-2$ up to an admissible terminal time $\tau\in I_n$.  The kinetic term on its left-hand side satisfies
\[
 \frac12\int_{\R^3}|u(x,\tau)|^2\Phi_n(x_3,\tau)\eta(x_3,\tau)\dd x
 \ge c r_n^{-1}\int_{S_n}|u(x,\tau)|^2\dd x.
\]
Taking the essential supremum over such $\tau$ controls the first term in $r_n^{-1}E_n$.  Next choose admissible terminal times $\tau_\ell\uparrow0$.  Since the dissipation integrand is nonnegative, monotone convergence gives
\begin{align*}
 \lim_{\ell\to\infty}
 \int_{-2}^{\tau_\ell}\int_{\R^3}|\nabla u|^2\Phi_n\eta\dd x\dd t
 &\ge c r_n^{-1}\int_{I_n}\int_{S_n}|\nabla u|^2\dd x\dd t.
\end{align*}
Thus the full left-hand side controls
\[
 cr_n^{-1}\left(
 \esssup_{t\in I_n}\int_{S_n}|u(x,t)|^2\dd x
 +\int_{\mathcal Q_n}|\nabla u|^2\dd x\dd t
 \right)=ce_n.
\]
\end{proof}

\begin{proposition}[Endpoint slab recursion]\label{prop:endpoint-recursion}
There are constants $C_0,C_1>0$, independent of $n$, such that $C_1$ depends only on the fixed exponents and cutoffs, while
\[
 C_0=C_0\!\left(p,m,E_{\mathrm g},
 \|u^3\|_{\widetilde L^p(-2,0;\dot B^0_{m,1})}\right),
\]
and
\begin{equation}\label{eq:endpoint-recursion}
 e_n
 \le C_0+C_1\sum_{j=0}^{n}
 \bigl(B_{j+c_*}+\widehat B_{j+c_*}\bigr)e_j,
 \qquad n\ge1.
\end{equation}
Moreover,
\begin{equation}\label{eq:endpoint-coefficient-tail}
 \lim_{N\to\infty}\sum_{j\ge N}
 \bigl(B_{j+c_*}+\widehat B_{j+c_*}\bigr)=0.
\end{equation}
\end{proposition}

\begin{proof}
Fix $n\ge1$ and an admissible terminal time $\tau\in I_n$.  By Lemma~\ref{lem:endpoint-horizontal-cutoff}, we may use
\[
 \phi_n:=\Phi_n\eta
\]
as a test function in the local energy inequality on $(-2,\tau)$.  Since both $\Phi_n$ and $\eta$ depend only on $(x_3,t)$ and $(\partial_t+\partial_3^2)\Phi_n=0$,
\begin{align*}
 (\partial_t+\Delta)(\Phi_n\eta)
 &=\Phi_n(\partial_t+\partial_3^2)\eta
   +2\partial_3\Phi_n\partial_3\eta,\\
 u\cdot\nabla(\Phi_n\eta)
 &=u^3\eta\partial_3\Phi_n
   +u^3\Phi_n\partial_3\eta.
\end{align*}
The right-hand side of the local energy inequality is bounded by
\[
 c\bigl(J_1+J_2+J_3+J_4+J_5\bigr),
\]
where
\begin{align*}
 J_1
 &:=\abs{\int_{-2}^{\tau}\int_{\R^3}|u|^2
 \left[\Phi_n(\partial_t+\partial_3^2)\eta
 +2\partial_3\Phi_n\partial_3\eta\right]\dd x\dd t},\\
 J_2
 &:=\abs{\int_{-2}^{\tau}\int_{\R^3}|u|^2u^3
 \Phi_n\partial_3\eta\dd x\dd t},\\
 J_3
 &:=\abs{\int_{-2}^{\tau}\int_{\R^3}\pi u^3
 \Phi_n\partial_3\eta\dd x\dd t},\\
 J_4
 &:=\abs{\int_{-2}^{\tau}\int_{\R^3}|u|^2u^3
 \eta\partial_3\Phi_n\dd x\dd t},\\
 J_5
 &:=\abs{\int_{-2}^{\tau}\int_{\R^3}\pi u^3
 \eta\partial_3\Phi_n\dd x\dd t}.
\end{align*}
The estimates below are uniform in $\tau$.

We first control $J_1,J_2,J_3$.  Derivatives of $\eta=\eta_0$ are supported in a fixed set on which either $-t$ is bounded away from zero or $|x_3|$ is bounded away from zero.  If $-t\ge c_0>0$, then $-t+r_n^2\ge c_0$ and the required derivatives of $\Phi_n$ are uniformly bounded.  If instead $|x_3|\ge c_0>0$, the Gaussian factor in the derivative bounds for $\Phi_n$ dominates every negative power of $-t+r_n^2$.  Hence, on the support of the derivatives of $\eta$,
\[
 |\Phi_n|+|\partial_3\Phi_n|+|\partial_3^2\Phi_n|
 \le c
\]
with a constant independent of $n$.  The global energy bound and
\[
 |u|^3+|\pi||u|\in L^1(\R^3\times(-2,0))
\]
therefore imply
\[
 J_1+J_2+J_3\le C_0,
\]
where $C_0$ is independent of $n$ and $\tau$.

For $J_4$, the uniform convection estimate \eqref{eq:endpoint-convection-flux} gives
\[
 J_4\le c\sum_{k=0}^nB_ke_{k-c_*}.
\]
For $J_5$, Proposition~\ref{prop:endpoint-pressure-flux} gives
\[
 J_5\le C
 +c\sum_{j=0}^n(B_{j+c_*}+\widehat B_{j+c_*})e_j.
\]
For $0\le k<c_*$, the negative-index convention gives $e_{k-c_*}=E_{\mathrm g}$, and these finitely many terms are absorbed into $C_0$.  For $k\ge c_*$, set $j=k-c_*$.  Since the terms are nonnegative,
\[
 \sum_{k=c_*}^nB_ke_{k-c_*}
 \le\sum_{j=0}^nB_{j+c_*}e_j.
\]
After increasing $C_0$ if necessary, the preceding bounds show, uniformly for admissible $\tau\in I_n$, that
\begin{align*}
 &r_n^{-1}\int_{S_n}|u(x,\tau)|^2\dd x
 +r_n^{-1}\int_{-2r_n^2}^{\tau}\int_{S_n}|\nabla u|^2\dd x\dd t\\
 &\qquad\le C_0+C_1\sum_{j=0}^n
 (B_{j+c_*}+\widehat B_{j+c_*})e_j.
\end{align*}
Taking the essential supremum over admissible $\tau\in I_n$ for the kinetic term and then a sequence $\tau_\ell\uparrow0$ for the dissipation term, as in Lemma~\ref{lem:endpoint-coercivity}, yields \eqref{eq:endpoint-recursion}.  Finally, \eqref{eq:endpoint-summability} implies
\[
 \sum_{j\ge0}(B_{j+c_*}+\widehat B_{j+c_*})<\infty,
\]
and hence its tail tends to zero.  This proves \eqref{eq:endpoint-coefficient-tail}.
\end{proof}

\begin{corollary}[Endpoint type-I slab energy]\label{cor:endpoint-energy}
Under \eqref{eq:endpoint-parameters} and $\sum_{j\in\Z}a_j<\infty$,
\begin{equation}\label{eq:endpoint-type-I}
 \sup_{n\ge1}e_n<\infty.
\end{equation}
\end{corollary}

\begin{proof}
Set
\[
 \lambda_j:=B_{j+c_*}+\widehat B_{j+c_*},
 \qquad j\ge0.
\]
By \eqref{eq:endpoint-coefficient-tail}, choose $N\ge1$ so large that
\begin{equation*}
 C_1\sum_{j\ge N}\lambda_j\le\frac12.
\end{equation*}
For each fixed $j$, $e_j<\infty$ because $E_j\le E_{\mathrm g}$ and $r_j>0$.  Hence
\[
 C_N:=C_0+C_1\sum_{j=0}^{N-1}\lambda_je_j<\infty.
\]
For every $n\ge N$, Proposition~\ref{prop:endpoint-recursion} therefore gives
\begin{equation*}
 e_n\le C_N+C_1\sum_{j=N}^{n}\lambda_je_j.
\end{equation*}

Fix $M\ge N$ and define
\[
 F_M:=\max_{N\le n\le M}e_n.
\]
Choose $n_M\in[N,M]$ such that $F_M=e_{n_M}$.  Applying
the recursion above with $n=n_M$ and using the nonnegativity of $\lambda_j$, we obtain
\begin{align*}
 F_M
 &\le C_N+C_1\sum_{j=N}^{n_M}\lambda_je_j\\
 &\le C_N+C_1F_M\sum_{j=N}^{n_M}\lambda_j\\
 &\le C_N+C_1F_M\sum_{j\ge N}\lambda_j
 \le C_N+\frac12F_M,
\end{align*}
where the last inequality follows from the choice of $N$ above.  Thus
\[
 F_M\le2C_N
\]
with a bound independent of $M$.  Letting $M\to\infty$ gives
\[
 \sup_{n\ge N}e_n\le2C_N.
\]
The finitely many values $e_1,\ldots,e_{N-1}$ are finite, so
\[
 \sup_{n\ge1}e_n<\infty.
\]
This proves \eqref{eq:endpoint-type-I}.
\end{proof}

\section{Compactness and proof of the main results}\label{sec:global}

For $r>0$, define the standard scale-invariant quantities
\begin{align}
 A(r)&:=r^{-1}\esssup_{-r^2<t<0}\int_{B_r}|u|^2\dd x,
 \label{eq:A-r}\\
 E(r)&:=r^{-1}\int_{Q_r}|\nabla u|^2\dd x\dd t,
 \label{eq:E-r}\\
 C(r)&:=r^{-2}\int_{Q_r}|u|^3\dd x\dd t,
 \label{eq:C-r}\\
 D(r)&:=r^{-2}\int_{Q_r}
 \abs{\pi-(\pi)_{B_r}(t)}^{3/2}\dd x\dd t.
 \label{eq:D-r}
\end{align}

For some $0<r_*\le1$, we shall use the scale-invariant type-I condition
\begin{equation}\label{eq:type-I-energy}
 \sup_{0<r<r_*}\bigl(A(r)+E(r)\bigr)<\infty.
\end{equation}

\begin{lemma}[Critical pressure control]\label{lem:pressure-critical}
Suppose \eqref{eq:type-I-energy} holds and $\pi\in L^{3/2}(Q_1)$.  Then
\begin{equation}\label{eq:critical-pressure}
 \sup_{0<r<r_*}D(r)<\infty.
\end{equation}
Moreover, for every fixed $\Lambda\ge1$,
\begin{equation}\label{eq:scaled-pressure-control}
 \sup_{0<r<r_*/\Lambda}
 \left\|
 r^2\bigl[\pi(rx,r^2t)-(\pi)_{B_r}(r^2t)\bigr]
 \right\|_{L^{3/2}(Q_\Lambda)}
 <\infty.
\end{equation}
\end{lemma}

\begin{proof}
\begin{revision}
Energy interpolation on $Q_r$ gives
\begin{equation}\label{eq:C-type-I}
 C(r)\le C\bigl(A(r)+E(r)\bigr)^{3/2},
\end{equation}
so $M_C:=\sup_{0<r<r_*}C(r)<\infty$.  We now record the pressure decay with all scale factors.  Fix $0<\theta\le1/4$ and choose $\chi\in C_c^\infty(B_r)$ such that $\chi=1$ on $B_{3r/4}$ and $|\nabla^\ell\chi|\le C_\ell r^{-\ell}$.  For almost every time set
\[
 \pi_0:=\sum_{i,j=1}^3\mathcal R_i\mathcal R_j
 \bigl[(u^iu^j)\chi\bigr],
 \qquad
 \pi_h:=\pi-(\pi)_{B_r}-\pi_0.
\]
Then $\pi_h$ is harmonic in $B_{r/2}$.  Calder\'on--Zygmund boundedness and the inclusion $Q_{\theta r}\subset Q_r$ yield
\begin{equation}\label{eq:pressure-local-part-detailed}
 (\theta r)^{-2}\int_{-(\theta r)^2}^0
 \int_{B_{\theta r}}|\pi_0|^{3/2}\,dx\,dt
 \le C\theta^{-2}C(r).
\end{equation}
For the harmonic part, the interior gradient estimate gives, for almost every $t$,
\[
 \sup_{B_{r/4}}|\nabla\pi_h(\cdot,t)|^{3/2}
 \le Cr^{-9/2}
 \int_{B_{r/2}}|\pi_h-(\pi_h)_{B_{r/2}}|^{3/2}\,dx.
\]
Poincar\'e's inequality in $B_{\theta r}$ therefore implies
\begin{align*}
 &\int_{B_{\theta r}}
 |\pi_h-(\pi_h)_{B_{\theta r}}|^{3/2}\,dx\\
 &\qquad\le C(\theta r)^{9/2}
 \sup_{B_{\theta r}}|\nabla\pi_h|^{3/2}\\
 &\qquad\le C\theta^{9/2}
 \int_{B_{r/2}}|\pi_h-(\pi_h)_{B_{r/2}}|^{3/2}\,dx.
\end{align*}
Since $\pi_h=\pi-(\pi)_{B_r}-\pi_0$, Jensen's inequality and the Calder\'on--Zygmund estimate give
\[
 \int_{B_{r/2}}|\pi_h-(\pi_h)_{B_{r/2}}|^{3/2}\,dx
 \le C\int_{B_r}|\pi-(\pi)_{B_r}|^{3/2}\,dx
 +C\int_{B_r}|u|^3\,dx.
\]
Integrating over $(-\theta^2r^2,0)$, enlarging the time interval to $(-r^2,0)$, and dividing by $(\theta r)^2$, we obtain
\begin{equation}\label{eq:pressure-harmonic-part-detailed}
 (\theta r)^{-2}\int_{Q_{\theta r}}
 |\pi_h-(\pi_h)_{B_{\theta r}}|^{3/2}
 \le C\theta^{5/2}\bigl(D(r)+C(r)\bigr).
\end{equation}
Combining \eqref{eq:pressure-local-part-detailed} and \eqref{eq:pressure-harmonic-part-detailed}, and absorbing the harmless term $\theta^{5/2}C(r)$ into $\theta^{-2}C(r)$, gives
\begin{equation}\label{eq:pressure-decay}
 D(\theta r)\le C\theta^{5/2}D(r)+C\theta^{-2}C(r).
\end{equation}
Choose $\theta_0\in(0,1/4)$ so that $C\theta_0^{5/2}\le1/2$.  For any fixed $0<r_0<r_*$, iteration gives
\[
 D(\theta_0^jr_0)
 \le2^{-j}D(r_0)+2C\theta_0^{-2}M_C,
 \qquad j\ge0.
\]
Here $D(r_0)<\infty$ follows from $\pi\in L^{3/2}(Q_1)$.  If
$\theta_0^{j+1}r_0<r\le\theta_0^jr_0=:R$, then Jensen's inequality for the two pressure averages gives
\[
 \int_{Q_r}|\pi-(\pi)_{B_r}|^{3/2}
 \le C\int_{Q_R}|\pi-(\pi)_{B_R}|^{3/2},
\]
and hence $D(r)\le C\theta_0^{-2}D(R)$ for every $0<r<r_0$.  For the remaining scales $r_0\le r<r_*$, another application of Jensen's inequality yields
\[
 D(r)
 \le Cr_0^{-2}\int_{Q_{r_*}}|\pi|^{3/2}\,dx\,dt<\infty.
\]
Together with the small-scale iteration, this proves \eqref{eq:critical-pressure}.

It remains to verify the pressure bound after rescaling to a fixed larger cylinder.  If $B_\rho\subset B_r$ and $r/\rho\le\Lambda$, Jensen's inequality yields
\begin{equation}\label{eq:pressure-average-comparison}
 |B_\rho|\,| (\pi)_{B_\rho}(t)-(\pi)_{B_r}(t)|^{3/2}
 \le C_\Lambda\int_{B_r}|\pi-(\pi)_{B_r}(t)|^{3/2}\,dx.
\end{equation}
Thus, for $0<r<r_*/\Lambda$,
\begin{align*}
 &\left\|r^2\bigl[\pi(rx,r^2t)-(\pi)_{B_r}(r^2t)\bigr]
 \right\|_{L^{3/2}(Q_\Lambda)}^{3/2}\\
 &\quad=r^{-2}\int_{Q_{\Lambda r}}
 |\pi-(\pi)_{B_r}(t)|^{3/2}\,dx\,dt\\
 &\quad\le C_\Lambda r^{-2}\int_{Q_{\Lambda r}}
 |\pi-(\pi)_{B_{\Lambda r}}(t)|^{3/2}\,dx\,dt\\
 &\quad=C_\Lambda\Lambda^2D(\Lambda r),
\end{align*}
which is uniformly bounded by \eqref{eq:critical-pressure}.  This proves \eqref{eq:scaled-pressure-control}.
\end{revision}
\end{proof}

\begin{lemma}[Rigidity of a zero-component limit]\label{lem:zero-component-rigidity}
Let $(V,\Pi)$ be a suitable weak solution in $Q_1$ and suppose that $V^3=0$ almost everywhere.  \rev{Then the origin is a regular point.  Equivalently, there exists $\rho_0\in(0,1/2)$ such that $V\in L^\infty(Q_{\rho_0})$.}
\end{lemma}

\begin{proof}
\begin{revision}
Apply \cite[Corollary~1.5]{ChaeWolf2021} at the origin with the auxiliary exponents
\[
 p_{\rm CW}=4,
 \qquad q_{\rm CW}=\infty,
 \qquad \frac2{p_{\rm CW}}+\frac3{q_{\rm CW}}=\frac12<1.
\]
Because $V^3\equiv0$, the one-component hypothesis of that result is trivially satisfied.  Hence the origin is regular, which gives $V\in L^\infty(Q_{\rho_0})$ for some $\rho_0\in(0,1/2)$.  This argument does not assert that $V$ is two-dimensional: the horizontal field may still depend on $x_3$ and satisfies
\[
 \partial_tV^h+V^h\cdot\nabla_hV^h
 -\Delta_hV^h-\partial_3^2V^h+\nabla_h\Pi=0,
 \qquad \nabla_h\cdot V^h=0.
\]
\end{revision}
\end{proof}

\begin{proposition}[Type-I compactness and vanishing component]\label{prop:compactness-rigidity}
Let $(u,\pi)$ be a suitable weak solution in a backward neighborhood of the origin.  After a fixed parabolic scaling, assume that $Q_1$ is contained in this neighborhood and that \eqref{eq:type-I-energy} holds for some $0<r_*\le1$.
Let $r_k=2^{-k}$ and define
\begin{equation}\label{eq:blowup-scaling}
 u^{(k)}(x,t):=r_k u(r_kx,r_k^2t),
 \qquad
 \pi^{(k)}(x,t):=r_k^2\bigl(\pi(r_kx,r_k^2t)-(\pi)_{B_{r_k}}(r_k^2t)\bigr).
\end{equation}
The subtracted spatial average depends only on time and therefore does not alter either the distributional equation or the local energy inequality.
If
\begin{equation}\label{eq:common-component-vanishing}
 (u^{(k)})^3\longrightarrow0
 \quad\text{in }\mathcal D'(Q_1),
\end{equation}
then the origin is regular.
\end{proposition}

\begin{proof}
Assume that the origin is singular.  Lemma~\ref{lem:pressure-critical} and \eqref{eq:type-I-energy} give the scale-invariant pressure bound.  Fix $R\ge1$.  Since the original solution is defined in $Q_1$, the rescaled pair is defined in $Q_{1/r_k}$; hence $Q_{2R}\subset Q_{1/r_k}$ for all sufficiently large $k$.  On this larger cylinder the type-I and pressure bounds yield
\[
 \sup_{k\ge k_0(R)}\left(
 \norm[L^\infty(-4R^2,0;L^2(B_{2R}))]{u^{(k)}}
 +\norm[L^2(-4R^2,0;H^1(B_{2R}))]{u^{(k)}}
 +\norm[L^{3/2}(Q_{2R})]{\pi^{(k)}}\right)<\infty.
\]
Local energy interpolation also gives a uniform $L^{10/3}(Q_{2R})$ bound.  Restricting the equation to $B_R\times(-R^2,0)$,
\[
 \partial_tu^{(k)}
 =\Delta u^{(k)}-\operatorname{div}(u^{(k)}\otimes u^{(k)})-\nabla\pi^{(k)},
\]
we find that the time derivatives are bounded in $L^{3/2}(-R^2,0;W^{-1,3/2}(B_R))$.  Indeed, the viscous term is bounded in $L^2_tH^{-1}_x$, the nonlinear term in $L^{5/3}_tW^{-1,5/3}_x$, and the pressure term in $L^{3/2}_tW^{-1,3/2}_x$.

Since
\[
 H^1(B_R)\Subset L^2(B_R)\hookrightarrow W^{-1,3/2}(B_R),
\]
the Aubin--Lions--Simon compactness theorem \cite{Simon1987} gives, after extraction, strong convergence in $L^2(Q_R)$.  Applying this argument for $R=1,2,\ldots$ and taking a diagonal subsequence, we obtain, for every fixed $R\ge1$,
\begin{align}
 u^{(k)}&\stackrel{*}{\rightharpoonup} U
 &&\text{in }L^\infty(-R^2,0;L^2(B_R)),\notag\\
 u^{(k)}&\rightharpoonup U
 &&\text{in }L^2(-R^2,0;H^1(B_R)),\notag\\
 u^{(k)}&\longrightarrow U
 &&\text{strongly in }L^2(Q_R),\notag\\
 \pi^{(k)}&\rightharpoonup\Pi
 &&\text{in }L^{3/2}(Q_R).\notag
\end{align}
The limit $U$ belongs to $L^{10/3}(Q_R)$, and $u^{(k)}-U$ is uniformly bounded in that space.  Hence interpolation gives
\begin{equation}\label{eq:compactness}
 \norm[L^3(Q_R)]{u^{(k)}-U}
 \le
 \norm[L^2(Q_R)]{u^{(k)}-U}^{1/6}
 \norm[L^{10/3}(Q_R)]{u^{(k)}-U}^{5/6}
 \longrightarrow0.
\end{equation}
Thus the strong $L^3$ convergence holds on the full terminal cylinder $Q_R$, not only on compact subsets separated from $t=0$.

\begin{revision}
The distributional equation passes to the limit because
\[
 u^{(k)}\otimes u^{(k)}\longrightarrow U\otimes U
 \quad\text{strongly in }L^{3/2}(Q_R),
\]
while the pressure converges weakly in $L^{3/2}$.  For a nonnegative compactly supported test function, the kinetic and convection terms in the local energy inequality converge by the strong $L^2$ and $L^3$ convergence, respectively; the pressure flux converges by pairing weak $L^{3/2}$ convergence of $\pi^{(k)}$ with strong $L^3$ convergence of $u^{(k)}$; and the dissipation is lower semicontinuous under weak $L^2$ convergence of the gradients.  Thus the local energy inequality passes to the limit, as in \cite{Lin1998}.  Since the diagonal subsequence is constructed on $Q_R$ for every integer $R\ge1$, $(U,\Pi)$ is a suitable ancient solution on $\mathbb R^3\times(-\infty,0)$.
\end{revision}  Condition \eqref{eq:common-component-vanishing} and \eqref{eq:compactness} imply $U^3=0$ almost everywhere in $Q_1$.  \rev{Lemma~\ref{lem:zero-component-rigidity} therefore gives a radius $\rho_0\in(0,1/2)$ such that $U\in L^\infty(Q_{\rho_0})$.}

It remains to retain singularity in the limit.  By the pressure-free criterion \cite[Theorem~1.2]{JiuWangZhou2019}, there is a universal $\eps_v>0$ such that a suitable weak solution $v$ in $Q_1$ is regular at the origin whenever
\[
 \int_{Q_1}|v|^{20/7}\dd x\dd t<\eps_v.
\]
For every sufficiently large $k$, the rescaled solution $u^{(k)}$ is singular at the origin.  Applying the contrapositive after rescaling $Q_\vartheta$ to $Q_1$ gives
\[
 \eps_v
 \le \vartheta^{-15/7}\int_{Q_\vartheta}|u^{(k)}|^{20/7}\dd x\dd t
 \le c\left(\vartheta^{-2}\int_{Q_\vartheta}|u^{(k)}|^3\dd x\dd t\right)^{20/21}.
\]
Consequently, for some universal $\eps_*>0$,
\begin{equation}\label{eq:singular-lower}
 \vartheta^{-2}\int_{Q_\vartheta}|u^{(k)}|^3\dd x\dd t\ge\eps_*,
 \qquad 0<\vartheta<1/4.
\end{equation}
\begin{revision}
For fixed $0<\vartheta<\min\{\rho_0,1/4\}$, the strong convergence \eqref{eq:compactness} on $Q_1$ passes this lower bound to $U$.  But
\[
 \vartheta^{-2}\int_{Q_\vartheta}|U|^3\dd x\dd t
 \le c\vartheta^3\norm[L^\infty(Q_{\rho_0})]{U}^3
 \longrightarrow0
 \qquad(\vartheta\downarrow0),
\]
a contradiction.
\end{revision}
\end{proof}

\begin{proposition}[Endpoint criterion implies regularity]\label{prop:endpoint-regularity}
Let $(u,\pi)$ be a suitable weak solution on $\R^3\times(-2,0)$ in the global energy class.  Assume \eqref{eq:endpoint-parameters} and
\begin{equation}\label{eq:endpoint-local-assumption}
 u^3\in\widetilde L^p(-2,0;\Bdot^0_{m,1}(\R^3)).
\end{equation}
Then the origin is a regular point.
\end{proposition}

\begin{proof}
Corollary~\ref{cor:endpoint-energy} gives \eqref{eq:endpoint-type-I}.  If $r_{n+1}<r\le r_n$, then $B_r\subset S_n$, $(-r^2,0)\subset I_n$, and $r^{-1}\le2r_n^{-1}$.  Hence
\[
 A(r)+E(r)\le2e_n,
\]
so \eqref{eq:type-I-energy} holds for some $r_*>0$.
The endpoint assumption and Lemma~\ref{lem:homogeneous-reconstruction} imply $u^3\in L^p(-2,0;L^m(\R^3))$.  Critical scaling gives
\begin{align}\label{eq:endpoint-component-vanishing}
 \norm[L^p(-1,0;L^m(B_1))]{(u^{(k)})^3}
 =\norm[L^p(-r_k^2,0;L^m(B_{r_k}))]{u^3}
 \longrightarrow0
\end{align}
by absolute continuity of the $L^p_tL^m_x$ norm.  Proposition~\ref{prop:compactness-rigidity} now applies.
\end{proof}

We now prove the global endpoint theorem, its critical-line consequence, and the continuation criterion.

\begin{proof}[Proof of Theorem~\ref{thm:main}]
Fix $z_0=(x_0,t_0)\in\R^3\times(0,T]$.  Choose a dyadic radius $r_0=2^{-N}$ such that $t_0-2r_0^2>0$, and define, for $-2<t<0$,
\begin{equation}\label{eq:endpoint-localized-scaling}
 v(x,t):=r_0u(x_0+r_0x,t_0+r_0^2t),
 \qquad
 \pi_v(x,t):=r_0^2\pi(x_0+r_0x,t_0+r_0^2t).
\end{equation}
The pair $(v,\pi_v)$ is a suitable weak solution in the global energy class on $\R^3\times(-2,0)$.  Translation commutes with the homogeneous blocks, while the dyadic dilation shifts their indices.  Since $1-2/p-3/m=0$,
\[
 \norm[\widetilde L^p(-2,0;\Bdot^0_{m,1})]{v^3}
 =\norm[\widetilde L^p(t_0-2r_0^2,t_0;\Bdot^0_{m,1})]{u^3}<\infty.
\]
Proposition~\ref{prop:endpoint-regularity} gives regularity of the normalized origin.  Reversing \eqref{eq:endpoint-localized-scaling} shows that $z_0$ is backward regular.  Since $z_0$ was arbitrary, the theorem follows.
\end{proof}

\begin{proof}[Proof of Corollary~\ref{cor:critical-line}]
Lemma~\ref{lem:critical-line-embedding} gives
\[
 \norm[\widetilde L^p(0,T;\Bdot^0_{m,1})]{u^3}
 \le c\norm[\widetilde L^p(0,T;\Bdot^s_{q,1})]{u^3}.
\]
The conclusion follows from Theorem~\ref{thm:main}.
\end{proof}

\begin{proof}[Proof of Corollary~\ref{cor:mild}]
Assume that $T^*<\infty$ and that the norm in \eqref{eq:blowup-norm} is finite for one admissible pair $(p,q)$.  By Lemma~\ref{lem:critical-line-embedding}, the endpoint norm $\norm[\widetilde L^p(0,T^*;\Bdot^0_{m,1})]{u^3}$ is finite.  \begin{revision}
We first verify directly that the mild solution belongs to the global energy class.  This avoids starting weak--strong uniqueness at a positive time before the two solution traces have been identified.  Fix $0<T<T^*$ and set
\[
 M_T:=\sup_{0\le t\le T}\|u(t)\|_{L^3}<\infty,
\]
which is finite because the maximal $L^3$-mild solution belongs to $C([0,T];L^3)$.  The mild formula and the heat-semigroup estimate
\[
 \bigl\|e^{(t-s)\Delta}\mathbb P\nabla\cdot F\bigr\|_{L^2}
 \le C(t-s)^{-3/4}\|F\|_{L^{3/2}}
\]
give, for $0<t\le T$,
\begin{align}
 \|u(t)-e^{t\Delta}u_0\|_{L^2}
 &\le C\int_0^t(t-s)^{-3/4}\|u(s)\otimes u(s)\|_{L^{3/2}}\,ds\notag\\
 &\le CM_T^2\int_0^t(t-s)^{-3/4}\,ds
 \le CM_T^2t^{1/4}.
 \label{eq:mild-L2-trace}
\end{align}
Since $e^{t\Delta}u_0\to u_0$ in $L^2$, it follows that
\[
 u(t)\longrightarrow u_0\quad\text{strongly in }L^2(\mathbb R^3)
 \qquad(t\downarrow0).
\]
For $0<\delta<t\le T$, restart the mild solution at time $\delta$.  Since $u(\delta)\in L^2\cap L^3$ and the positive-time solution is smooth, the standard whole-space energy equality (obtained, for instance, by the usual spatial-cutoff approximation) yields
\[
 \|u(t)\|_{L^2}^2
 +2\int_\delta^t\|\nabla u(s)\|_{L^2}^2\,ds
 =\|u(\delta)\|_{L^2}^2.
\]
Letting $\delta\downarrow0$ and using \eqref{eq:mild-L2-trace}, we obtain
\begin{equation}\label{eq:mild-energy-equality}
 \|u(t)\|_{L^2}^2
 +2\int_0^t\|\nabla u(s)\|_{L^2}^2\,ds
 =\|u_0\|_{L^2}^2,
 \qquad 0<t<T^*.
\end{equation}
Consequently,
\[
 u\in L^\infty(0,T^*;L^2(\mathbb R^3))
 \cap L^2(0,T^*;\dot H^1(\mathbb R^3)).
\]
Energy interpolation gives $u\in L^{10/3}(\mathbb R^3\times(0,T^*))$, and the whole-space pressure
\[
 \pi=\sum_{i,j=1}^3\mathcal R_i\mathcal R_j(u^iu^j)
\]
belongs to $L^{5/3}(\mathbb R^3\times(0,T^*))$.  Since the solution is smooth on every compact positive-time interval, it satisfies the local energy equality there; hence $(u,\pi)$ is a suitable weak solution on $\mathbb R^3\times(0,T^*)$.  Applying Theorem~\ref{thm:main} with terminal time $T^*$ shows that every point $(x_0,T^*)$ is backward regular.
\end{revision}

A spatially uniform terminal bound follows from energy interpolation:
\begin{equation}\label{eq:global-ten-thirds}
 u\in L^{10/3}(\R^3\times(0,T^*)),
 \qquad
 \pi=\sum_{i,j=1}^3\mathcal R_i\mathcal R_j(u^iu^j)
 \in L^{5/3}(\R^3\times(0,T^*)),
\end{equation}
where the pressure is taken in the whole-space normalization.  Therefore,
\begin{equation}\label{eq:space-tail}
 \int_0^{T^*}\int_{\{|x|>R\}}
 \bigl(|u|^{10/3}+|\pi|^{5/3}\bigr)\dd x\dd t
 \longrightarrow0
 \qquad (R\to\infty).
\end{equation}
Fix
\[
 0<\rho_*<\min\left\{1,\sqrt{T^*/2}\right\}.
\]
If $|x_0|>R+2\rho_*$, then $Q_{\rho_*}(x_0,T^*)$ is contained in the spatial tail in \eqref{eq:space-tail}.  H\"older's inequality and Jensen's inequality for the pressure average give
\begin{align}\label{eq:far-field-epsilon}
 &\rho_*^{-2}\int_{Q_{\rho_*}(x_0,T^*)}|u|^3\dd x\dd t
 \notag\\[-1mm]
 &\quad+\rho_*^{-2}\int_{Q_{\rho_*}(x_0,T^*)}
 \abs{\pi-(\pi)_{B_{\rho_*}(x_0)}(t)}^{3/2}\dd x\dd t
 \notag\\
 &\qquad\le c_{\rho_*}
 \left(
   \int_0^{T^*}\int_{\{|x|>R\}}|u|^{10/3}\dd x\dd t
 \right)^{9/10}
 \notag\\
 &\qquad\quad+c_{\rho_*}
 \left(
   \int_0^{T^*}\int_{\{|x|>R\}}|\pi|^{5/3}\dd x\dd t
 \right)^{9/10}.
\end{align}
Choose $R$ so large that the right-hand side is below the universal $\varepsilon$-regularity threshold in \cite{CKN1982,Lin1998}.  The scale-one quantitative conclusion of that criterion gives a constant $c_{\mathrm{far}}$, independent of $x_0$, such that
\begin{equation}\label{eq:far-field-bound}
 \norm[L^\infty(Q_{\rho_*/2}(x_0,T^*))]{u}
 \le c_{\mathrm{far}} \rho_*^{-1},
 \qquad |x_0|>R+2\rho_*.
\end{equation}

\begin{revision}
For each $x\in\overline{B_{R+3\rho_*}}$, backward regularity provides a radius $\rho_x>0$ and a number $M_x<\infty$ such that
\[
 \|u\|_{L^\infty(B_{\rho_x}(x)\times(T^*-\rho_x^2,T^*))}\le M_x.
\]
The balls $B_{\rho_x/2}(x)$ cover the compact set $\overline{B_{R+3\rho_*}}$.  Choose a finite subcover associated with $x_1,\ldots,x_N$ and set
\[
 \tau_0:=\frac14\min_{1\le i\le N}\rho_{x_i}^2,
 \qquad
 c_{\mathrm{near}}:=\max_{1\le i\le N}M_{x_i}.
\]
Then
\begin{equation}\label{eq:compact-terminal-bound}
 \|u\|_{L^\infty(B_{R+3\rho_*}\times(T^*-\tau_0,T^*))}
 \le c_{\mathrm{near}}.
\end{equation}
\end{revision}
Combining \eqref{eq:far-field-bound} and \eqref{eq:compact-terminal-bound}, and decreasing the terminal time interval if necessary, gives
\begin{equation}\label{eq:global-terminal-Linfty}
 u\in L^\infty\bigl(\R^3\times(T^*-\tau,T^*)\bigr)
 \quad\text{for some }\tau>0.
\end{equation}
For a related localization argument, compare the proof of \cite[Theorem~1.2]{ChaeWolf2021}.

The global energy inequality and \eqref{eq:global-terminal-Linfty} imply
\[
 \norm[L^3(\R^3)]{u(t)}^3
 \le\norm[L^\infty(\R^3)]{u(t)}\norm[L^2(\R^3)]{u(t)}^2
 \le c,
 \qquad T^*-\tau<t<T^*.
\]
On every earlier compact time interval, the mild solution is continuous in $L^3$.  Hence $u\in L^\infty(0,T^*;L^3)$, contradicting the $L^3$ blow-up alternative for maximal mild solutions; see \cite{Kato1984,ESS2003}.  Therefore \eqref{eq:blowup-norm} must hold.
\end{proof}

\medskip

\section*{Acknowledgement}
W. Wang was supported by National Key R\&D Program of China (No.2023YFA1009200) and NSFC under grant 12471219.  

{\bibliographystyle{plain}}
\bibliography{references}

\end{document}